\documentclass[10pt]{amsart}

\usepackage{amsfonts,amssymb,amscd,amstext}
\usepackage[charter]{mathdesign}
\usepackage[colorlinks=true,linkcolor=blue,citecolor=red]{hyperref}
\usepackage[utf8]{inputenc}

\renewcommand{\le}{\leqslant}
\renewcommand{\ge}{\geqslant}
\newcommand{\ptl}{\partial}
\newcommand{\rr}{{\mathbb{R}}}
\newcommand{\la}{\lambda}
\newcommand{\hh}{H}
\newcommand{\h}{H}
\newcommand{\esf}{\mathbb{S}}

\newcommand{\nn}{\mathbb{N}}
\newcommand{\pp}{P}
\newcommand{\ppc}{P_C}
\newcommand{\escpr}[1]{\big<#1\big>}

\newcommand{\eps}{\varepsilon}

\newcommand{\de}{\delta}

\newcommand{\vol}[1]{|#1|}

\newcommand{\cl}[1]{\text{\rm cl}(#1)}
\newcommand{\clb}{\mkern2mu\overline{\mkern-2mu B\mkern-2mu}\mkern2mu}

\DeclareMathOperator{\divv}{div}

\DeclareMathOperator{\dist}{dist}
\DeclareMathOperator{\intt}{int}
\DeclareMathOperator{\Lip}{Lip}

\DeclareMathOperator{\inr}{inr}

\DeclareMathOperator{\diam}{diam}

\newtheorem{theorem}{Theorem}[section]
\newtheorem{proposition}[theorem]{Proposition}
\newtheorem{lemma}[theorem]{Lemma}
\newtheorem{corollary}[theorem]{Corollary}

\theoremstyle{definition}

\newtheorem{remark}[theorem]{Remark}

\theoremstyle{remark}

\newenvironment{enum}{\begin{enumerate}
}{\end{enumerate}}

\numberwithin{equation}{section}

\begin{document}

\title{Isoperimetric inequalities in Euclidean convex bodies}

\author[M.~Ritor\'e]{Manuel Ritor\'e} \address{Departamento de
Geometr\'{\i}a y Topolog\'{\i}a \\
Universidad de Granada \\ E--18071 Granada \\ Espa\~na}
\email{ritore@ugr.es}
\author[S.~Vernadakis]{Efstratios Vernadakis} \address{Departamento de
Geometr\'{\i}a y Topolog\'{\i}a \\
Universidad de Granada \\ E--18071 Granada \\ Espa\~na}
\email{stratos@ugr.es}

\date{\today}

\thanks{Both authors have been supported by MICINN-FEDER grant MTM2010-21206-C02-01, and Junta de Andaluc\'{\i}a grants FQM-325 and P09-FQM-5088}

\dedicatory{Dedicated to Carlos Ben\'{\i}tez on his 70th birthday}

\begin{abstract}
In this paper we consider the problem of minimizing the relative perimeter under a volume constraint in the interior of a convex body, i.e., a compact convex set in Euclidean space with interior points. We shall not impose any regularity assumption on the boundary of the convex set. Amongst other results, we shall prove the equivalence between Hausdorff and Lipschitz convergence, the continuity of the isoperimetric profile with respect to the Hausdorff distance, 
and the convergence in Hausdorff distance of sequences of isoperimetric regions and their free boundaries. We shall also describe the behavior of the isoperimetric~profile for small volume, and the behavior of isoperimetric regions for small volume.
\end{abstract}


\maketitle

\thispagestyle{empty}

\bibliographystyle{amsplain}

\section{Introduction}
In this work we consider the \emph{isoperimetric problem} of minimizing perimeter under a given volume constraint inside a \emph{convex body}, a 	\emph{compact} convex set $C\subset\rr^{n+1}$ with interior points. The perimeter considered here will be the one relative to the interior of $C$. No regularity assumption on the boundary will be assumed. This problem is often referred to as the \emph{partitioning problem}.

A way to deal with this problem is to consider the \emph{isoperimetric profile} $I_C$ of $C$, i.e., the function assigning to each $0<v<|C|$ the infimum of the relative perimeter of the sets inside $C$ of volume $v$. The isoperimetric profile can be interpreted as an optimal isoperimetric inequality in $C$. A minimum for this problem will be called an \emph{isoperimetric region}. The \emph{normalized isoperimetric profile} $J_C$ is defined on the interval $(0,1)$ by $J_C(\la)=I_C(\la |C|)$.

The isoperimetric profile of convex bodies with smooth boundary has been intensively~considered, and many results are known, such as the concavity of the isoperimetric profile, Sternberg and Zumbrun \cite{MR1674097}, the concavity of the $\big(\tfrac{n+1}{n}\big)$ power of the isoperimetric profile, Kuwert \cite{MR2008339}, the connectedness of the reduced boundary of the isoperimetric regions \cite{MR1674097}, the behavior of the isoperimetric profile for small volumes, B\'erard and Meyer \cite{be-me}, or the behavior of isoperimetric regions for small volumes, Fall \cite{fall}. See also \cite{bayle}, \cite{bay-rosal} and \cite{MR1803220}. The results in all these papers make a strong use of the regularity of the boundary. In particular, in \cite{MR1674097} and \cite{MR2008339}, the $C^{2,\alpha}$ regularity of the boundary implies a strong regularity of the isoperimetric regions up to the boundary, except in a singular set of large Hausdorff codimension, that allows the authors to apply the classical first and second variation formulas for volume and perimeter. The convexity of the boundary then implies the concavity of the profile and the connectedness of the regular part of the free boundary.

Up to our knowledge, the only known results for non-smooth boundary are the ones by Bokowski and Sperner \cite{bo-sp} on isoperimetric inequalities for the Minkowski content in Euclidean convex bodies, the isoperimetric inequality for convex cones by Lions and Pacella \cite{lions-pacella} using the Brunn-Minkowski inequality, with the characterization of isoperimetric regions by Figalli and Indrei \cite{FI}, the extension of Levy-Gromov inequality, \cite[App.~C]{grom}, to arbitrary convex sets given by Morgan \cite{MR2438911}, and the extension of the concavity of the $\big(\tfrac{n+1}{n}\big)$ power of the isoperimetric profile to arbitrary convex bodies by E.~Milman \cite[\S~6]{MR2507637}. In his work on the isoperimetric profile for small volumes in the \emph{boundary} of a polytope, Morgan mentions that his techniques can be adapted to handle the case of small volumes in a solid polytope, \cite[Remark~3.11]{morganpolytops}, without uniqueness, see Remark after Theorem~3.8 in \cite{morganpolytops}.  We recall that isoperimetric inequalities outside a convex set with smooth boundary have been obtained in \cite{MR2338131}, \cite{MR2215458}, \cite{MR2329803}. Previous estimates on least perimeter in convex bodies have been obtained by Dyer and Frieze \cite{MR1141926}, Kannan, Lov\'asz and Simonovits \cite{MR1318794} and Bobkov \cite{MR2347041}. In the initial stages of this research the authors were greatly influenced by the paper of Bokowski and Sperner \cite{bo-sp}, see also \cite{MR936419}. This work is divided into two different parts: in the first one the authors characterize the isoperimetric regions in a ball (for the Minkowski content) using spherical symmetrization, see also \cite{MR934771} and \cite{MR2590630}. In the second part, given a convex body $C$ so that there is a closed ball $\clb(x,r)\subset C$, they build a map between $\clb(x,r)$ and $C$, which transform the volume and the perimeter in a controlled way, allowing them to transfer the isoperimetric inequality of the ball to $C$. This map is not bilipschitz, but can be  modified to satisfy this property.

In this paper we extend some of the results already known for Euclidean convex bodies with smooth boundary to arbitrary convex bodies, and prove new results for the isoperimetric profile.  We begin by considering the Hausdorff and Lipschitz convergences in the space of convex bodies. We prove in Theorem~\ref{thm:lipschitz} that a sequence $C_i$ of convex bodies that converges to a convex body $C$ in Hausdorff distance also converges in Lipschitz distance. This is done by considering a ``natural'' sequence of bilipschitz maps $f_i:C\to C_i$, defined by \eqref{eq:fi}, and proving that $\Lip(f_i)$, $\Lip(f_i^{-1})\to 1$. These maps are modifications of the one used by Bokowski and Sperner in \cite{bo-sp} and have the following key property, see Corollary~\ref{cor:estilipconst}: if $\clb(0,2r)\subset C\cap C'$, $C\cup C'\subset \clb(0,R)$ and $f:C\to C'$ is the considered map then $\Lip(f)$, $\Lip(f^{-1})$ are bounded above by a constant depending only on $R/r$. This implies, see Theorem~\ref{thm:bo-sp}, a uniform non-optimal isoperimetric inequality for all convex bodies with bounded quotient circumradius/inradius. We also prove in Theorem~\ref{thm:equiv} that Lipschitz convergence implies convergence in the weak Hausdorff topology (modulo isometries). Let us recall that in Bayle's Ph.D. Thesis \cite[Thm.~4.2.7]{bayle}  was proven the convergence of the isoperimetric profiles of a sequence of Riemannian manifolds in $\mathcal{M}(n,d,v,\delta)$ converging in Gromov-Hausdorff distance to a Riemannian manifold in the same class. Here $\mathcal{M}(n,d,v,\delta)$ denotes the set of compact $n$-dimensional Riemannian manifolds satisfying $\text{diam}(M,g)\le d$, $\text{vol}(M,g)\ge v$, and $\text{Ricci}_{(M,g)}\ge (n-1)\,\delta\,g$. Let us also recall that the Gromov compactness theorem \cite{grom} implies that the space of compact $n$-dimensional Riemannian manifolds $(M,g)$ with sectional curvatures satisfying $|K|\le c_1$, $\text{vol}(M,g)\le c_2$ and $\text{diam}(M,g)\le c_3$ is precompact in the Lipschitz topology, see also \cite{MR892147}, \cite{MR917868}. Results proving the convergence of the boundaries of smooth non-compact convex hypersurfaces have been given by Alexander and Ghomi \cite{MR2019227}. 

Using Theorem~\ref{thm:lipschitz} we prove in Theorem~\ref{thm: is cont} the pointwise convergence of the normalized isoperimetric profiles. This implies, Corollary~\ref{cor:concavidad de perfil}, through approximation by smooth convex bodies, the concavity of the isoperimetric profile $I_C$ and of the function $I_C^{(n+1)/n}$ for an arbitrary convex body. As observed by Bayle \cite[Thm.~2.3.10]{bayle}, the concavity of $I_C^{(n+1)/n}$ implies the strict concavity of $I_C$. This is an important property that implies the connectedness of an isoperimetric region and of its complement, Theorem~\ref{thm:connectedness}. By standard properties of concave functions, we also obtain in Corollary~\ref{cor:J_{C_i}toJ_C_unif} the uniform convergence of the normalized isoperimetric profiles $J_C$, and of their powers $J_C^{(n+1)/n}$ in compact subsets of the interval $(0,1)$. Using the bilipschitz maps constructed in the first section, we show in Theorem~\ref{thm:isnqgdbl} that a uniform relative isoperimetric inequality, and hence a Poincar\'e inequality, holds in metric balls of small radius in $C$.

Using this relative isoperimetric inequality we prove in Theorem~\ref{thm:leon rigot lem 42} a key result on the density of an isoperimetric region and its complement, similar to the ones obtained by Leonardi and Rigot \cite{le-ri}, which are in fact based on ideas by David and Semmes \cite{MR1625982} for quasi-minimizers of the perimeter. Theorem~\ref{thm:leon rigot lem 42} is closer to a ``clearing out'' result as in Massari and Tamanini \cite[Thm.~1]{MR1124566} (see also \cite{MR1943988}) than to a concentration type argument as in Morgan's \cite[\S~13.7]{MR2455580}. One of the consequences of Theorem~\ref{thm:leon rigot lem 42} is a uniform lower density result,  Corollary~\ref{cor:monotonicity}. The estimates obtained in Theorem~\ref{thm:leon rigot lem 42} are stable enough to allow passing to the limit under Hausdorff convergence. Hence we can improve the $L^1$ convergence of isoperimetric regions and show in Theorem~\ref{thm:EitoE Haus} that this convergence is in Hausdorff distance (see \cite[\S~1.3]{tam-reg} and \cite[Thm.~2.4.5]{MR1736268}). We can prove the convergence of the free boundaries in Hausdorff distance in Theorem~\ref{thm:ptl haus} as well. As a consequence, we are able to show in Theorem~\ref{thm: ptl conect} that, given a convex body $C$, for every $0<v<\vol{C}$, there always exists an isoperimetric region with connected free boundary.

Finally, in the last section we consider the isoperimetric profile for small volumes. In the smooth boundary case, Fall \cite{fall} showed that for sufficiently small volume, the isoperimetric regions are small perturbations of geodesic spheres centered at a global maximum of the mean curvature, and derived an asymptotic expansion for the isoperimetric profile. We show in Theorem~\ref{thm:optinsmalvol} that the isoperimetric profile of a convex set for small volumes is asymptotic to the one of its smallest tangent cone, i.e., the one with the smallest solid angle, and that rescaling isoperimetric regions to have volume $1$ makes them subconverge in Hausdorff distance to an isoperimetric region in this convex cone, which is a geodesic ball centered at some apex by the recent result of Figalli and Indrei \cite{FI}. Although in the interior of the convex set we can apply Allard's regularity result for rectifiable varifolds, obtaining high order convergence of the boundaries of isoperimetric sets, we do not dispose of any regularity result at the boundary to ensure convergence up to the boundary (unless both the set and its limit tangent cone have smooth boundary \cite{MR863638}). As a consequence of Theorem~\ref{thm:optinsmalvol}, we show in Theorem~\ref{thm:polytops} that the only isoperimetric regions of sufficiently small volume inside a convex polytope are geodesic balls centered at the vertices whose tangent cones have the smallest solid angle. The same result holds when the convex set is locally a cone at the points of the boundary with the smallest solid angle. A similar result for the \emph{boundary} of the polytope was proven by Morgan \cite{morganpolytops}.

We have organized this paper into several sections. In the next one we introduce the basic background and notation. In the third one we shall consider the relation between  the Hausdorff and Lipschitz convergence for convex bodies. In the fourth one we shall prove the continuity of the isoperimetric profile with respect to the Hausdorff distance and some consequences, in the fifth one we shall prove the density result and the convergence of isoperimetric regions and their free boundaries in Hausdorff distance. In the last section, we shall study the behavior of the isoperimetric profile and of the isoperimetric regions for small volume.

The results in this paper are intended to be applied to study the behavior of the asymptotic isoperimetric profile of unbounded convex bodies (closed unbounded convex sets with non-empty interior) in Euclidean space.

The authors would like to thank Frank Morgan and Gian Paolo Leonardi for their helpful suggestions and comments.

\section{Preliminaries}

Throughout this paper we shall denote by $C\subset\rr^{n+1}$ a compact convex set with non-empty interior. We shall call such a set a \emph{convex body}. 
Note that this terminology does not agree with some classical texts such as Schneider \cite{sch}. As a rule, basic properties of convex sets which are stated without proof in this paper can be easily found in Schneider's monograph.

The Euclidean distance in $\rr^{n+1}$ will be denoted by $d$, and the $r$-dimensional Hausdorff measure of a set $E$ by $\hh^r(E)$. The volume of a set $E$ is its $(n+1)$-dimensional Hausdorff measure and we shall denote it by $\vol{E}$. We shall denote the closure of $E$ by $\cl{E}$ or $\overline{E}$ and the topological boundary by $\ptl E$. The open ball of center $x$ and radius $r>0$ will be denoted by $B(x,r)$, and the corresponding closed ball by $\clb(x,r)$.

In the space of convex bodies one may consider two different notions of convergence. Given a convex body $C$, and $r>0$, we define $C_r=\{p\in\rr^{n+1}: d(p,C)\le r\}$. The set $C_r$ is the tubular neighborhood of radius $r$ of $C$ and is a closed convex set. Given two convex sets $C$, $C'$, we define its \emph{Hausdorff distance} $\delta(C,C')$ by
\begin{equation}
\label{eq:hd}
\delta(C,C')=\inf\{r>0: C\subset (C')_r, C'\subset C_r\}.
\end{equation}
The space of convex bodies with the Hausdorff distance is a metric space. Bounded sets in this space are relatively compact by Blaschke's Selection Theorem, \cite[Thm.~1.8.4]{sch}. We shall say that a sequence $\{C_i\}_{i\in\nn}$ of convex bodies converges to a convex body $C$ in Hausdorff distance if $\lim_{i\to\infty}\delta(C_i,C)=0$.

Given two convex bodies $C$, $C'\subset\rr^{n+1}$, we define its \emph{weak Hausdorff distance} $\de_S(C,C')$~by
\begin{equation}
\label{eq:shd}
\de_S(C,C')=\inf\{\de(C,h(C')): h\in\text{Isom}(\rr^{n+1})\}.
\end{equation}
The weak Hausdorff distance is non-negative, symmetric, and satisfies the triangle inequality. Moreover, $\de_S(C,C')=0$ if and only if there exists $h\in\text{Isom}(\rr^{n+1})$ such that $C=h(C')$.

A map $f:(X,d)\to (X',d')$ between metric spaces is \emph{lipschitz} if~there exists a constant $L>0$ so that
\begin{equation}
\label{eq:lipschitzdef}
d'(f(x),f(y))\le L\,d(x,y),
\end{equation}
for all $x$, $y\in X$. Sometimes we will refer to such a map as an $L$-lipschitz map. The smallest constant satisfying \eqref{eq:lipschitzdef}, sometimes called the dilatation of $f$, will be denoted by $\Lip(f)$. A lipschitz function on $(X,d)$ is a lipschitz map $f:X\to\rr$, where we consider on $\rr$ the Euclidean distance. A map $f:X\to Y$ is bilipschitz if both $f$ and $f^{-1}$ are lipschitz maps.

Given two convex bodies $C$, $C'$, we define its \emph{Lipschitz distance} $d_L$ by
\begin{equation}
\label{eq:ld}
d_L(C,C')=\inf_{f\in\Lip(C,C')}\{\log(\max\{\Lip(f),\Lip(f^{-1})\})\},
\end{equation}
where $\Lip(C,C')$ is the set of bilipschitz maps from $C$ to $C'$. We shall say that a sequence $\{C_i\}_{i\in\nn}$ of convex bodies converges in Lipschitz distance to a convex body $C$ if $\lim_{i\to\infty} d_L(C_i,C)=0$. The Lipschitz distance is non-negative, symmetric and satisfies the triangle inequality. Moreover, $d_L(C,C')=0$ if and only if $C$ and $C'$ are isometric. If a sequence $\{C_i\}_{i\in\nn}$ converges to $C$ is the lipschitz sense, then there is a sequence of bilipschitz maps $f_i:C_i\to C$ such that
\[
\lim_{i\to\infty} \log(\max\{\Lip(f_i),\Lip(f_i^{-1})\})=0.
\]
This implies $\lim_{i\to\infty}\max\{\Lip(f_i),\Lip(f_i^{-1})\}=1$. As $1\le \Lip(f_i)\Lip(f_i^{-1})$, we obtain that both $\Lip(f_i)$, $\Lip(f_i^{-1})\to 1$. Conversely, if there is a sequence of bilipschitz maps $f_i:C_i\to C$ such that $\lim_{i\to\infty}\Lip(f_i)=\lim_{i\to\infty}\Lip(f_i^{-1})=1$ then $\lim_{i\to\infty}d_L(C_i,C)=0$.

If $M$, $N$ are subsets of Euclidean spaces and $f: M\to N$ is a lipschitz map, then $g:\la M\to\la N$ defined by $g(x)=\la f(\tfrac{x}{\la}), x \in \la M, \la >0$, is a lipschitz map so that $\Lip(g)=\Lip(f)$. This yields the very useful consequence
\begin{equation}
\label{eq:lipinvomo}
d_L(\la M,\la N)=d_L(M,N), \qquad \la>0.
\end{equation}

For future reference, we list the following properties of lipschitz maps and functions

\begin{lemma}\mbox{}
\label{lem:lipschitz}
\begin{enum}
\item Let $f$ be a lipschitz function on $(X,d)$ so that $|f|\ge M>0$. Then $1/f$ is a lipschitz function and $\Lip(1/f)\le \Lip(f)/M^2$.
\item Let $f_1, f_2$ be lipschitz functions on $(X,d)$. Then $f_1+f_2$ is a lipschitz function and $\Lip(f_1+f_2)\le\Lip(f_1)+\Lip(f_2)$.
\item Let $f_1, f_2$ be lipschitz functions on $(X,d)$ so that $|f_i|\le M_i$, $i=1,2$. Then $f_1f_2$ is a lipschitz function and $\Lip(f_1f_2)\le M_1\Lip(f_2)+M_2\Lip(f_1)$.
\item If $\la :(X,d)\to \rr$ is lipschitz with $|\lambda|\le L'$, and $f:(X,d)\to \rr^n$ is lipschitz with $|f|<M'$, then $\Lip(\la f)\le M'\Lip(\lambda)+L'\Lip(f)$.
\item If $f_i$ are lipschitz maps that converge pointwise to a lipschitz map $f$, then $\Lip(f)\le\liminf_{i\to\infty}\Lip(f_i)$.
\end{enum}
\end{lemma}

The behavior of the Hausdorff measure \cite[\S~1.7.2]{bbi} with respect to lipschitz maps is well known.

\begin{lemma}
\label{lem:bilip}
Let $C,C'\subset\rr^{n+1}$ and $f:C\to C'$ a Lipschitz map. Then, for every $s>0$ and $E\subset C$ we have
\begin{equation}
\h^s(f(E)) \le \Lip(f)^s\, \h^s(E).
\end{equation}
Morever, If $f$ is bilipschitz then we have
\begin{equation}
\frac{1}{\Lip(f ^{-1})^s}\,\h^s(E) \le \h^s(f(E)) \le \Lip(f)^s\, \hh^s(E).
\end{equation}
\end{lemma}

For $t\ge 0$, let $E(t)$ denote the set of points of density $t$ of $E$ in $C$
\[
E(t)=\{ x\in C: \lim_{r\to 0}\frac{\vol{E\cap B_C(x,r)}}{\vol{B_C(x,r)}}=t \}.
\]
Since $\vol{E\cap\ptl C}=0$, we have that
$\vol{E(t)}=\vol{E(t)\cap\intt(C)}$. By Lebesgue- Besicovitch Theorem we have
$\vol{E(1)}=\vol{E}$ and similarly $\vol{E(0)}=\vol{C\setminus E}$. 

For $E \subset C$, we define the \emph{perimeter} of $ E$ in the \emph{interior} of $C$ by
\[
\pp_C(E)=\pp(E,\intt(C))= \sup \Big \{ \int_E\divv \xi\, d{\h}^{n+1}, \xi \in \mathfrak{X}_0(\intt(C)) ,\, |\xi| \le 1 \Big \},
\]
where $\mathfrak{X}_0(\intt(C))$ is the set of smooth vector fields with compact support in the interior of $C$. We shall say that $E$ has \emph{finite perimeter} in $\intt(C)$ if $P_C(E)<\infty$. A set $E$ of finite perimeter in $\intt(C)$ satisfies $P(E)\le P_C(E)+H^n(\ptl C)$ and so is a Cacciopoli set in $\rr^{n+1}$. We can define its reduced boundary $\ptl^*E$ as in \cite[Chapter~3]{gi} and we have $P_C(E)=H^n(\ptl^*E\cap\intt(C))$.

Observe that we are only taking into account the $\mathcal{H}^n$-measure of $\ptl E$ inside the interior of $C$. We define the \emph{isoperimetric profile} of $C$ by
\begin{equation}
\label{eq:profile}
I_C(v)=\inf \Big \{ \pp_C(E) : E \subset C, \vol{E} = v \Big \}.
\end{equation}
We shall say that $E \subset C$ is an \emph{isoperimetric region} if $\pp_C(E)=I_C(|E|)$. The \emph{renormalized isoperimetric profile} of $C$ is
\begin{equation}
\label{eq:renprofile}
Y_C=I_C^{(n+1)/n}.
\end{equation}
We shall denote by $J_C:[0,1]\to\rr^+$ the \emph{normalized isoperimetric profile} function
\begin{equation}
\label{eq:jc}
J_C(\la)=I_C(\la\,\vol{C}).
\end{equation}
We shall also denote by $y_C:[0,1]\to\rr^+$ the function
\begin{equation}
\label{eq:def:y_C}
y_C=J_C^{(n+1)/n}.
\end{equation}

Standard results of Geometric Measure Theory imply that isoperimetric regions exist in a convex body. The following basic properties are well known.

\begin{lemma}
\label{lem:fundamental}
Let $C\subset\rr^{n+1}$ be a convex body. Consider a sequence $\{E_i\}_{i\in\nn}\subset C$ of subsets with finite perimeter in the interior of $C$.
\begin{enum}
\item
If  $E_i$ converges to a set $ E\subset C$ with finite perimeter in $\intt(C)$ in the $L^1(\intt(C))$ sense, then $\pp_C(E)\le\liminf_{i\to\infty}\ppc(E_i) $

\item If  $\ppc(E_i)$ is uniformly bounded from above, then there exists a set $E\subset C$ of finite perimeter in $\intt(C)$ such that a subsequence of $\{E_i\}_{i\in\nn}$  converges to $E$ in the $L^1(\intt(C))$ sense.
\item Isoperimetric regions exist in $C$ for every volume.
\item $I_C$ is continuous.
\end{enum}
\end{lemma}

\begin{proof}
Properties (i), (ii) and (iii) follow from the lower semicontinuity of perimeter \cite[Thm.~1.9]{gi} and  compactness \cite[Thm.~1.19]{gi}. The continuity of the isoperimetric profile was proven in \cite[Lemma~6.2]{gallot}.
\end{proof}

For a convex body $C$, the continuity of the isoperimetry profile of $C$ will be a trivial consequence of the concavity of $I_C$ proven in Corollary~\ref{cor:concavidad de perfil}.

The known results on the regularity of isoperimetric regions are summarized in the following Lemma.

\begin{lemma}[{\cite{MR684753}, \cite{MR862549}, \cite[Thm.~2.1]{MR1674097}}]
\label{lem:n-7}
\mbox{}
Let $C\subset \rr^{n+1}$ a convex body and $E\subset C$ an isoperimetric region.
Then $\ptl E\cap\intt(C) = S_0\cup S$, where  $S_0\cap S=\emptyset$ and
\begin{enum}
\item $S$ is an embedded $C^{\infty}$ hypersurface of constant mean curvature.\item $S_0$ is closed and $H^{s}(S_0)=0$ for any $s>n-7$.
\end{enum}
Moreover, if the boundary of $C$ is of class $C^{2,\alpha}$ then $\cl{\ptl E\cap\intt(C)}=S\cup S_0$, where
\begin{enum}
\item[(iii)] $S$ is an embedded $C^{2,\alpha}$ hypersurface of constant mean curvature
\item[(iv)] $S_0$ is closed and $H^s(S_0)=0$ for any $s>n-7$
\item[(v)] At points of $ S \cap \ptl C$,  $ S$ meets $\ptl C$ orthogonally.
\end{enum}
\end{lemma}

\section{Hausdorff and Lipschitz convergence in the space of convex bodies}

As a first step in our study of the isoperimetric profile of a convex body, we need to~prove that Hausdorff convergence of convex bodies implies Lipschtz convergence. We shall also prove the converse replacing the Hausdorff distance by the weak Hausdorff distance as defined in \eqref{eq:shd}. We need first some preliminary results for convex sets.

Given a convex body $C\subset\rr^n$ containing $0$ in its interior, its \emph{radial function} $\rho(C,\cdot):\esf^n\to\rr$ is defined by
\[
\rho(C,u)=\max\{\la\ge 0:\la u\in C\}.
\]
From this definition it follows that $\rho(C,u)u\in \ptl C$ for all $u\in\esf^n$.

\begin{lemma}
\label{lem:radial}
Let $C\subset\rr^{n+1}$ be a convex body so that $B(0,r)\subset C\subset B(0,R)$. Then the radial function $\rho(C,\cdot):\esf^n\to\rr$ is $R^2/r$-lipschitz.
\end{lemma}

\begin{proof}
Let $C^*$ be the polar body of $C$, \cite[\S~1.6]{sch}. Theorem~1.6.1 in \cite{sch} implies that $(C^*)^*=C$ and that $B(0,1/R)\subset C^*\subset B(0,1/r)$. Let $h(C^*,\cdot)$ be the support function of $C^*$. Using $(C^*)^*=C$, Remark~1.7.7 in \cite{sch} implies
\[
\rho(C,u)=\frac{1}{h(C^*,u)}.
\]
By Lemma~1.8.10 in \cite{sch} the function $h(C^*,\cdot)$ is $1/r$-lipschitz. Since $h(C^*,\cdot)\ge 1/R$, we conclude from Lemma~\ref{lem:lipschitz} that $\rho(C,\cdot)$ is an $R^2/r$-lipschitz function.
\end{proof}

\begin{lemma}
\label{lem:radialconv}
Let $\{C_i\}_{i\in\nn}$ be a sequence of convex bodies converging in Hausdorff distance to a convex body $C$. We further assume that there exist $r$, $R>0$ such that $B(0,r)\subset \intt(C_i)\subset B(0,R)$ for all $i\in\nn$, and $B(0,r)\subset \intt(C)\subset B(0,R)$. Then
\[
\lim_{i\to\infty} \sup_{u\in\esf^n} |\rho(C_i,u)-\rho(C,u)|=0.
\]
\end{lemma}

\begin{proof}
We reason by contradiction. Assume there exists $\eps>0$ and $u_i\in\esf^n$ so that a subsequence satisfies
\[
|\rho(C_i,u_i)-\rho(C,u_i)|\ge\eps.
\]
Passing again to a subsequence we may assume that $u_i\to u\in\esf^n$. We define
\[
x_i=\rho(C_i,u_i)\,u_i\in\ptl C_i,\qquad y_i=\rho(C,u_i)\,u_i\in\ptl C.
\]
Since $\rho(C_i,\cdot)$ and $\rho(C,\cdot)$ are uniformly bounded, we may extract again convergent subsequences $x_i\to x$ and $y_i\to y$. Since $\ptl C$ is closed, we have $y\in\ptl C$. Since $C_i\to C$ in Hausdorff distance, we have $x\in\ptl C$ (it is straightforward to check that $x\not\in \rr^{n+1}\setminus C$, and that $x\not\in\intt(C)$ by Lemma~1.8.14 in \cite{sch}). Since $|x_i-y_i|\ge\eps$ we get $|x-y|\ge\eps$, but both $x$, $y$ belong to the ray emanating from $0$ with direction $u$. This is a contradiction since $0\in\intt(C)$, \cite[Lemma~1.1.8]{sch}.
\end{proof}

\begin{lemma}
\label{lem:Lip(f_i-f) to 0}
Let $\{f_i\}_{i\in\nn}$ be a sequence of convex functions defined on a convex open set $C$ and converging uniformly on $C$ to a convex function $f$.
\begin{enum}
\item  Let $\{x_i\}_{i\in\nn}$ be a sequence such that $x=\lim_{i\to\infty} x_i$. If $\nabla f_i(x_i)$, $\nabla f(x)$ exist for all $i\in\nn$, then $\nabla f_i(x_i) \to \nabla f(x)$.
\item $\Lip(f_i-f)\to 0$.
\item If $g$ is a convex function defined in a convex body $C$, then
\[
\Lip(g) =\sup_{z\in D}|\nabla g(z)|,
\]
where $D$ is the subset of $C$ (dense and of full measure) where $ \nabla g$ exists.
\end{enum}
\end{lemma}
\begin{proof}
The proof of (i) is taken from \cite[Thm.~25.7]{roc}. We give it for completeness. Assume that $\nabla f_i(x_i)$ does not converge to $\nabla f(x)$. Then there exists $y\in\rr^n$ and $\eps >0$ such that either
\begin{equation}
\label{eq:eps}
\begin{split}
\escpr{\nabla f_i(x_i),y}-\escpr{\nabla f(x),y}&\ge \eps, \ \text{or}
\\
\escpr{\nabla f_i(x_i),y} -\escpr{\nabla f(x),y} &\le -\eps,
\end{split}
\end{equation}
holds for a subsequence.

Let us assume that the second inequality in \eqref{eq:eps} holds for a subsequence. For simplicity, we assume it holds for the whole sequence. Thus we have $\escpr{\nabla f_i(x_i),y} \le\escpr{\nabla f(x),y}  -\eps$ for any index $i$. Multiplying this inequality by $t<0$ we obtain $\escpr{\nabla f_i(x_i),ty}\ge \big(\escpr{\nabla f(x),y}-\eps\big)\,t$. From this inequality and the convexity of $f_i$ we get
\[
f_i(x_i+ty)-f_i(x_i)\ge \escpr{\nabla f_i(x_i),ty} \ge\big( \escpr{f(x),y}-\eps\big)\,t.
\]
Letting $i\to \infty$, taking into account that $f_i \to f$ uniformly, we find
\[
\frac{f(x+ty)-f(x)}{t}\le \langle \nabla f(x),y \rangle  -\eps
\]
Taking limits when $t\uparrow 0$ we get $\escpr{\nabla f(x),y}\le\escpr{\nabla f(x),y}-\eps$, and we reach a contradiction. The case of the first inequality in \eqref{eq:eps} is treated in the same way. This proves (i).

To prove (ii) we also reason by contradiction. So we assume there exists $\eps>0$ so that $\Lip(f_i-f)>\eps$ holds for a subsequence. For simplicity, we assume that every index $i$ satisfies this inequality. We can find sequences $\{x_i\}_{i\in\nn}$, $\{y_i\}_{i\in\nn}$ such that $x_i\neq y_i$ and
\begin{equation}
\label{eq:|(f_i-f)(x_i)-(f_i-f)(y_i)|}
|(f_i-f)(x_i)-(f_i-f)(y_i)|>\eps\, |x_i-y_i|\quad\text{for all}\ i\in\nn.
\end{equation}
Passing again to a subsequence if necessary, we assume that there are points $x$, $y$ such that $x=\lim_{i\to\infty}x_i$, $y=\lim_{i\to\infty}y_i$.

We observe that it can be assumed that both $\nabla f_i$ and $\nabla f$ are defined $\hh^1$-almost everywhere in the segment $[x_i,y_i]$: otherwise we consider a right circular cylinder $D\times [x_i,y_i]$ of axis $[x_i,y_i]$ so that, in every segment parallel to $[x_i,y_i]$ of height $|x_i-y_i|$, inequality \eqref{eq:|(f_i-f)(x_i)-(f_i-f)(y_i)|} is satisfied by its extreme points. Since the set where the gradients $\nabla f_i$, $\nabla f$ exist has full $\hh^{n+1}$-measure in $D\times [x_i,y_i]$, \cite[Thm.~25.4]{roc}, Fubini's Theorem implies that $\hh^n$-almost everywhere in $D$, the gradients are $\hh^1$-almost everywhere defined. We replace $[x_i,y_i]$ by one of such segments if necessary.

For $\la\in [0,1]$, and $i\in\nn$, we define convex functions $u_i,v_i$ by
\begin{equation}
 u_i(\la):=\frac{f_i(x_i+\la(y_i-x_i))-f_i(x_i)}{|y_i-x_i|},\qquad
 \, v_i(\la):=\frac{f(x_i+\la(y_i-x_i))-f(x_i)}{|y_i-x_i|}.
\end{equation}
Hence \eqref{eq:|(f_i-f)(x_i)-(f_i-f)(y_i)|} is equivalent to
\begin{equation}
\label{eq:u_i(1)-v_i(1)}
\lim_{i\to\infty}(u_i(1)-v_i(1)) \ge\eps
\end{equation}
We easily find
\begin{equation}
(u_i(\la)-v_i(\la))'= f_i{'}(x_i+\la(y_i-x_i);\frac{x_i-y_i}{|x_i-y_i|})-f'(x_i+\la(y_i-x_i);\frac{x_i-y_i}{|x_i-y_i|}),
\end{equation}
where the derivative $f'(p;u)$ of the convex function $f$ at the point $p$ in the direction of $u$ is defined as in \cite[p.~213]{roc}. At the points where both $\nabla f_i,\nabla f$ exist we get
\begin{equation*}
\big(u_i(\la)-v_i(\la)\big)'=\escpr{(\nabla f_i-\nabla f)(x_i+\la(y_i-x_i),\frac{x_i-y_i}{|x_i-y_i|}},
\end{equation*}
and
\begin{equation*}
|(u_i(\la)-v_i(\la))'|\le|\nabla f_i(x_i+\la(y_i-x_i))-\nabla f(x_i+\la(y_i-x_i))|.
\end{equation*}
By (i) and \cite[Thm.~25.5]{roc} we have $\lim_{i\to\infty}(u_i(\la)-v_i(\la))'=0$. By \cite[Thm.~10.6]{roc}, $\Lip(f_i)$ is uniformly bounded. So $(u_i-v_i)'$ is  bounded by a constant by (iii). 
Then by 
the Dominated Convergence Theorem, \cite[Corollary~24.2.1]{roc}, and the fact that $u_i(0)=v_i(0)=0$, we get
\[
\lim_{i\to \infty}(u_i(1)-v_i(1))=\lim_{i\to \infty}\int_0^1(u_i(\la)-v_i(\la))'d\la=0,
\]
which, together with \eqref{eq:u_i(1)-v_i(1)}, gives a contradiction. Hence $\lim_{i\to\infty}\Lip(f_i-f)=0$.

To prove (iii), let $z\in D$. There is $w\in \esf^n$ such that $|\nabla g(z)|=\langle \nabla g(z),w \rangle$. Hence
\[
|\nabla g(z)|=\Big|\lim_{\la\to 0}\frac{g(z+\la w)-g(z)}{\la}\Big|\le \sup_{x\not = y}\frac{|g(x)-g(y)|}{|x-y|}=\Lip(g).
\]
To prove the reverse inequality, take $x,y\in C$ and assume for the moment that $\nabla g$ exists $\hh^1$-almost everywhere in the segment $[x,y]$. Then by \cite[Corollary~24.2.1]{roc} we have
\[
|g(x)-g(y)|=\Big|\int_0^1\langle\nabla g(x+\la(y-x),y-x \rangle d\la\Big| \le \sup_{z\in D}|\nabla g(z)||x-y|
\]
If  $\nabla g$ does not exist $\hh^1$-almost everywhere in the segment $[x,y]$, we can make an approximation argument, as in the proof of (ii), with segments parallel to $[x,y]$, where $\nabla g$ exists $\hh^1$- almost everywhere, to conclude the proof.
\end{proof}

Now we prove that Hausdorff convergence of a sequence of convex bodies implies Lipschitz convergence.

\begin{theorem}
\label{thm:lipschitz}
Let $\{C_i\}_{i\in\nn}$ be a sequence of convex bodies in $\rr^{n+1}$ that converges in Hausdorff distance to a convex body $C$. Then $\{C_i\}_{i\in\nn}$ converges to $C$ in Lipschitz distance.
\end{theorem}

\begin{proof}
Translating the whole sequence and its limit we assume that $0\in\intt(C)$. Let $r>0$ so~that $\clb(0,2r)\subset\intt(C)$. By \cite[Lemma~1.8.14]{sch}
and the convergence of $C_i$ to $C$ in Hausdorff distance, there exists $i_0\in\nn$ such that $\clb(0,r)\subset\intt(C_i)$ for $i\ge i_0$. Let us denote by $\rho_i$ and $\rho$ the radial functions $\rho(C_i,\cdot)$ and $\rho(C,\cdot)$, respectively. Since the sequence $\{C_i\}_{i\in\nn}$ converges to $C$ in Hausdorff distance, there exists $R>0$ so that $\bigcup_{i\in\nn}C_i\cup C\subset B(0,R)$.

For $i\ge i_0$, we define a map $f_i:C\to C_i$ by
\begin{equation}
\label{eq:fi}
f_i(x)=\begin{cases}
x, & |x|\le r,
\\
\displaystyle r\,\frac{x}{|x|}+(|x|-r)\,\frac{\rho_i\big(\frac{x}{|x|}\big)-r}{\rho\big(\frac{x}{|x|}\big)-r}\,\frac{x}{|x|}, & |x|\ge r.
\end{cases}
\end{equation}
Using Lemmata~\ref{lem:lipschitz} and \ref{lem:radial} we obtain that $f_i$ is a lipschitz function. The inverse mapping can be defined exchanging the roles of $\rho_i$ and $\rho$ to conclude that $f_i$ is a bilipschitz map. The function $f_i$ can be rewritten as
\begin{equation}
\label{eq:fi2}
f_i(x)=x+\Big( 1-\frac{\rho_i\big(\frac{x}{|x|}\big)-r}{\rho\big(\frac{x}{|x|}\big)-r} \Big)\,(r-|x|)\,\frac{x}{|x|},\qquad|x|\ge r.
\end{equation}

To show that the sequence $\{C_i\}_{i\in\nn}$ converges in Lipschitz distance to $C$, it is enough to prove that both $\Lip(f_i)$, $\Lip(f_i^{-1})$ converge to $1$. We shall show that
\begin{equation}
\label{eq:limrhoi}
\lim_{i\to\infty}\Lip\Big( 1-\frac{\rho_i\big(\frac{x}{|x|}\big)-r}{\rho\big(\frac{x}{|x|}\big)-r} \Big)=0,
\end{equation}
and the corresponding inequality interchanging $\rho_i$ and $\rho$. From \eqref{eq:limrhoi} and the expression of $f_i$ given by \eqref{eq:fi2} we would get $\limsup_{i\to\infty}\Lip(f_i)\le 1$. Since $\Lip(f_i)\ge\Lip(f_i|_{\clb(0,r)})=1$ we obtain $1\le\liminf_{i\to\infty}\Lip(f_i)$. Crossing both inequalities we would have $\lim_{i\to\infty}\Lip(f_i)=1$. The same argument would work for $f_i^{-1}$.

Let us now prove \eqref{eq:limrhoi}. In what follows we shall assume that $\rho, \rho_i$ have $\esf^n$ as their~domain of definition. As $\rho-r$ is bounded from below, again by Lemma~\ref{lem:lipschitz}, it is enough to prove $\lim_{i\to\infty}\Lip(\rho_i-\rho)=0$. Let us denote by $h_i^*,h^*$ the support functions of the polar sets $C_i^*,C^*$ of $C_i,C$, respectively. By \cite[Remark~1.7.7]{sch}, $h_i^*=1/\rho_i$. Since $\rho_i$ is uniformly bounded from below, again by Lemma~\ref{lem:lipschitz}, it is enough to check that that $\Lip(h_i^*-h^*)\to 0$. By Lemma~\ref{lem:radialconv}, the convex functions $h_i^*$ converge pointwise to $h^*$. Lemma~\ref{lem:Lip(f_i-f) to 0} then implies that $\Lip(h_i^*-h^*)=0$.
\end{proof}

\begin{remark}
Observe that the map given by \eqref{eq:fi} is defined in all of $\rr^{n+1}$ and takes $C$ onto $C_i$ and $\rr^{n+1}\setminus C$ onto $\rr^{n+1}\setminus C_i$.
\end{remark}

\begin{remark}
\label{rem:lipschitz}
If $f:C_1\to C_2$ is a bilipschitz map between convex bodies of $\rr^{n+1}$, then $g:\la C_1\to \la C_2$, defined by $g(x)=\la f(\frac{x}{\la})$,  is also bilipschitz and satisfies $\Lip(f)=\Lip(g)$, $\Lip(f^{-1})=\Lip(g^{-1})$.
\end{remark} 

\begin{remark}
Let $C$, $C'\subset\rr^{n+1}$ two convex bodies so that $\delta(C,C')>0$, $d_L(C,C')>0$ (it is enough to consider two non-isometric convex bodies). For $i\in\nn$, we have
\[
d_L(iC,iC')=d_L(i^{-1}C,i^{-1}C')=d_L(C,C').
\]
On the other hand
\[
\delta(iC,iC')=i\,\delta(C,C')\to +\infty;\qquad \delta(i^{-1}C,i^{-1}C')=i^{-1}\delta(C,C')\to 0.
\]
Hence Lipschitz and Hausdorff distances will not be equivalent in a subset of the space of convex bodies unless we impose uniform bounds on the circumradius and the inradius.
\end{remark}

Now we prove that the convergence of a sequence of convex bodies in Lipschitz distance, together with an upper bound on the circumradii of the elements of the sequence, implies the convergence of a subsequence in Hausdorff distance to a convex body isometric to the Lipschitz limit. We recall that Lipschitz convergence implies Gromov-Hausdorff convergence, see \cite[Prop.~3.7]{grom}, \cite[Ex.~7.4.3]{bbi}. 

\begin{theorem}
\label{thm:equiv}
Let $\{C_i\}_{i\in\nn}$ be a sequence of convex bodies converging to a convex body $C$ in Lipschitz distance. Then $\{C_i\}_{i\in\nn}$ converges to $C$ in weak Hausdorff distance.
\end{theorem}

\begin{proof}
Let $f_i:C\to C_i$ be a sequence of bilipschitz maps with $\Lip(f_i)$, $\Lip(f_i^{-1})\to 1$. Then $\text{diam}(C_i)$ are uniformly bounded, so that translating the sets $C_i$ we may assume they are uniformly bounded. Applying the Arzel\`a-Ascoli Theorem, a subsequence of $f_i$ uniformly converges to a lipschitz map $f:C\to\rr^{n+1}$. We shall assume the whole sequence converges. The sequence $C_i=f_i(C)$ converges to the compact set $f(C)$ in the sense of Kuratowski \cite[Def.~4.4.13]{amb} and so converges to $f(C)$ in Hausdorff distance by \cite[Prop.~4.4.14]{amb}. To check that $C_i$ converges to $f(C)$ in the sense of Kuratowski we take $x=\lim_{k\to\infty} f_{i_k}(x_{i_k})$, with  $x_{i_k}\in C$, and we extract a convergent subsequence of $x_{i_k}$ to some $x_0\in C$ to get $x=f(x_0)\in f(C)$; on the other hand, every $x\in f(C)$ is the limit of the sequence of points $f_i(x)\in C_i$.

Since $f_i\to f$ and $\Lip(f_i)\to 1$, Lemma \ref{lem:lipschitz}  implies $\Lip(f)\le 1$ and $|f(x)-f(y)|\le |x-y|$ for any $x,y\in C$. On the other hand, taking limits when $i\to\infty$ in the inequalities
\[
|x-y|=|f_i^{-1}(f_i(x))-f_i^{-1}(f_i(y))|\le\Lip(f_i^{-1})\,|f_i(x)-f_i(y)|
\]
we get $|x-y|\le |f(x)-f(y)|$ and so $f$ is an isometry. This arguments shows that any subsequence of $\{C_i\}_{i\in\nn}$ has a convergent subsequence in weak Hausdorff distance to $C$, which is enough to conclude that $\lim_{i\to\infty}\de_S(C_i,C)=0$.
\end{proof}

In the next result we shall obtain a geometric upper bound for the lipschitz constant of the  map built in the proof of Theorem~\ref{thm:lipschitz}. Observe that the the same bound holds for the inverse mapping, which satisfies the same geometrical condition.

\begin{corollary}
\label{cor:estilipconst}
Let $C$, $C'\subset\rr^{n+1}$ be convex bodies so that $\clb(0,2r)\subset C\cap C'$, $C\cup C'\subset \clb(0,R)\subset \rr^{n+1}$. Let $f:C \to C'$ be the bilipschitz map defined by
\begin{equation}
\label{eq:f}
f(x)=\begin{cases}
x, & |x|\le r,
\\
\displaystyle r\,\frac{x}{|x|}+(|x|-r)\,\frac{\rho'\big(\frac{x}{|x|}\big)-r}{\rho\big(\frac{x}{|x|}\big)-r}\,\frac{x}{|x|}, & |x|\ge r.
\end{cases}
\end{equation}
Then we have
\begin{equation}
\label{eq:lipff-1}
1\le\Lip(f), \Lip(f^{-1})\le 1+ \frac{R}{r}\Big(\frac{R}{r}-1\Big) \Big(\frac{R^2}{r^2}+1 \Big).
\end{equation}
\end{corollary}

\begin{proof}
By Lemma \ref{lem:lipschitz} we get $\Lip(f)\ge \Lip(f|_{\{|x|\le r\}})=1$ and the same argument is valid for $f^{-1}$ as well. So in what is follows we assume that $|x|\ge r$. Observe that $x\in\rr^{n+1}\setminus B(0,r)\mapsto r\frac{x}{|x|}$ is the metric projection onto the convex set $\{|x|\le r\}$ and so has Lipschitz constant 1, thus
\begin{equation}
\label{eq:lipx/|x|}
\Lip\big(\frac{x}{|x|}\big)\le 1/r.
\end{equation}
We denote by $\rho, \rho'$ the radial functions of $C$, $C'$ respectively. Let us estimate first the Lipschitz constant of the map
\[
x\in\rr^{n+1}\setminus B(0,r)\mapsto\frac{\rho'\big(\frac{x}{|x|}\big)-r}{\rho\big(\frac{x}{|x|}\big)-r}.
\]
By Lemma \ref{lem:lipschitz} (i), (iii),(vii), and \eqref{eq:lipx/|x|} we get
\begin{equation}
\label{eq:corestilipconst1}
\Lip\Big( \frac{\rho'\big(\frac{x}{|x|}\big)-r}{\rho\big(\frac{x}{|x|}\big)-r} \Big)
\le \frac{1}{r}\frac{R^2}{r}\frac{1}{r}+(R-r)\frac{R^2}{r}\frac{1}{r}\frac{1}{r}=\frac{R^2}{r^3}+(R-r)\frac{R^2}{r^4}.
\end{equation}
As the above function is bounded from above by $\frac{R-r}{r}$, and $x\mapsto\frac{x}{|x|}$ is bounded from above by 1, having Lipschitz constant no larger than $1/r$ by \eqref{eq:lipx/|x|},  Lemma~\ref{lem:lipschitz} (iv) then implies
\begin{equation}
\label{eq:corestilipconst2}
\Lip\Big( \frac{\rho'\big(\frac{x}{|x|}\big)-r}{\rho\big(\frac{x}{|x|}\big)-r} \Big)\frac{x}{|x|}
\le \frac{R^2}{r^3}+(R-r)\frac{R^2}{r^4}+\frac{R-r}{r}\frac{1}{r}.
\end{equation}
Thus, as the above function is bounded from above by $\frac{R-r}{r}$, and $ x\mapsto |x|-r$ is bounded from above by $R-r$, having Lipschitz constant no larger than $1$, then from Lemma~\ref{lem:lipschitz} (iv) we get

\begin{equation}
\label{eq:corestilipconst3}
\begin{split}
\Lip(f) &\le 1+ (R-r)\Big(\frac{R^2}{r^3}+(R-r)\frac{R^2}{r^4}+\frac{R-r}{r^2}\Big)+\frac{R-r}{r}
\\
&\le 1+ \Big(\frac{R-r}{r}\Big)\Big(\frac{R^2}{r^2}+\Big(\frac{R-r}{r}\Big)\frac{R^2}{r^2}+\frac{R-r}{r}+1 \Big)
\\
&\le 1+ \Big(\frac{R}{r}-1\Big) \Big(\frac{R^3}{r^3}+\frac{R}{r} \Big).
\end{split}
\end{equation}
\end{proof}

\section{The isoperimetric profile in the space of convex bodies}

Using the results of the previous Section, we shall prove in this one that, when a sequence of convex bodies converges in Hausdorff distance to a convex body, then the normalized isoperimetric profiles defined by \eqref{eq:jc} and \eqref{eq:def:y_C} converge uniformly to the normalized isoperimetric profiles of the limit convex body. This has some consequences: the isoperimetric profile $I_C$ of a convex body $C$, and its power $I_C^{(n+1)/n}$, even with non-smooth boundary, are concave. This would imply that isoperimetric regions and their complements are connected, and also the connectedness of the free boundaries when the boundary is of class $C^{2,\alpha}$.

\begin{theorem}
\label{thm: is cont}
Let $\{C_i\}_{i\in\nn}$ be a sequence of convex bodies in $\rr^{n+1}$ that converges to a convex body $C\subset\rr^{n+1}$ in Hausdorff distance. Then $J_{C_i}$ converges to $J_C$ pointwise in $[0,1]$. Consequently, also $y_{C_i}$ converges pointwise to $y_C$.
\end{theorem}

\begin{proof}
For $\la\in\{0,1\}$ we have $J_{C_i}(\la)=J_C(\la)=0$. Let us fix some $\la\in (0,1)$. Let $\{E_i\}_{i\in\nn}$ be a sequence of isoperimetric regions in $C_i$ with $\vol{E_i}=\la\,\vol{C_i}$, see Lemma~\ref{lem:fundamental}. By the regularity lemma~\ref{lem:n-7}, $\pp_C(E_i)=\hh^n(\ptl E_i\cap\intt(C_i))$. By the continuity of the volume with respect to the Hausdorff distance, we have $\lim_{i\to\infty}\vol{E_i}=\la\,\vol{C}$.

Theorem ~\ref{thm:lipschitz} implies the existence of a sequence of bilipschitz maps $f_i: C_i\to C$ so that $\lim_{i\to\infty}\Lip(f_i)=\lim_{i\to\infty}\Lip(f_i)^{-1}=1$. Lemma~\ref{lem:bilip} yields
\begin{align*}
\frac{1}{\Lip(f_i^{-1}) ^{n+1}}\,\vol{E_i} &\le  \vol{f_i(E_i)}\le \Lip(f_i)^{n+1} \,\vol{E_i},
\\
\frac{1}{\Lip(f_i^{-1}) ^{n}}\,\pp_{C_i}(E_i) &\le \pp_C (f_i(E_i))\le \Lip(f_i)^n \,\pp_{C_i}(E_i).
\end{align*}
So $\{f_i(E_i)\}_{i\in\nn}$ is a sequence of finite perimeter sets in $C$ with $\lim_{i\to\infty}\vol{f_i(E_i)}=\la\,\vol{C}$, and $\liminf_{i\to\infty} P_{C_i}(E_i)=\liminf_{i\to\infty} P_C(f_i(E_i))$. From Lemma \ref{lem:fundamental} 
we have
\begin{align*}
J_C(\la)&\le\lim_{i\to\infty} I_C(\vol{f_i(E_i)})\le\liminf_{i\to\infty} P_C(f_i(E_i))
\\
&=\liminf_{i\to\infty} P_{C_i}(E_i)=\liminf_{i\to\infty} J_{C_i}(\la).
\end{align*}

Let us prove now that $J_C(\la)\ge \limsup_{i\to\infty} J_{C_i}(\la)$. We shall reason by contradiction assuming that $J_C(\la)<\limsup J_{C_i}(\la)$. Passing to a subsequence we can suppose that $\{J_{C_i}(\la)\}_{i\in\nn}$ converges. So let us assume $J_C(\la)<\lim_{i\to\infty} J_{C_i}(\la)$. Let $E\subset C$ be an isoperimetric region with $\vol{E}=\la\,\vol{C}$. Consider a point $p$ in the regular part of $\ptl E\cap\intt(C)$. We take a vector field in $\rr^{n+1}$ with compact support in a small neighborhood of $p$ that does not intersect the singular set of $\ptl E$. We choose the vector field so that the deformation $\{E_t\}_{t\in\rr}$ induced by the associated flow strictly increases the volume in the interval $(-\eps,\eps)$, i.e., $t\mapsto\vol{E_t}$ is strictly increasing in $(-\eps,\eps)$. Taking a smaller $\eps$ if necessary, the first variation formulas of volume and perimeter imply the existence of a constant $M>0$ so that
\begin{equation}
\label{eq:almgren}
|\hh^n(\ptl E_t\cap\intt(C))-\hh^n(\ptl E\cap\intt(C)|\le M\,|\vol{E_t}-\vol{E}|
\end{equation}
holds for all $t\in (-\eps,\eps)$. Reducing $\eps$ again if necessary we may assume
\begin{equation}
\label{eq:almcons}
\hh^n(\ptl E\cap\intt(C))+M\,|\vol{E_t}-\vol{E}|<\lim_{i\to\infty}J_{C_i}(\la).
\end{equation}
(recall we are supposing $\hh^n(\ptl E\cap\intt(C))=J_C(\la)<\lim_{i\to\infty} J_{C_i}(\la))$.

For every $i\in\nn$, consider the sets $\{f_i^{-1}(E_t)\}_{t\in (-\eps,\eps)}$. Since
\[
\frac{1}{\Lip(f_i) ^{n+1}}\,\vol{E_t} \le \vol{f_i^{-1}(E_t)}\le \Lip(f_i^{-1})^{n+1}\,\vol{E_i},
\]
$\vol{E_{-\eps/2}}<\la\,\vol{C}$, $\vol{E_{\eps/2}}>\la\,\vol{C}$ by the monotonicity of the function $t\mapsto\vol{E_t}$ in $(-\tfrac{\eps}{2},\tfrac{\eps}{2})$, the Lipschitz constants $\Lip(f_i), \Lip(f_i^{-1})$ converge to $1$ when $i\to\infty$, and $\lim_{i\to\infty} \vol{C_i}/\vol{C}=1$, there exists $i_0\in\nn$ such that
\[
\vol{f_i^{-1}(E_{\eps/2})}>\la\,\vol{C_i},\qquad  \vol{f_i^{-1}(E_{-\eps/2})}<\la\,\vol{C_i},
\]
for all $i\ge i_0$. Since $t\mapsto \vol{f_i^{-1}(E_t)}$ is continuous, for every $i\ge i_0$, there exists $t(i)\in (-\tfrac{\eps}{2},\tfrac{\eps}{2})$ so that $\vol{f_i^{-1}(E_{t(i)})}=\la\,\vol{C_i}$, and we have
\begin{align*}
\pp_{C_i}(f_i^{-1}(E_{t(i)}))&\le \Lip(f_i^{-1})\,\pp_C(E_{t(i)})
\\
&\le \Lip(f_i^{-1})\,\big(P_C(E)+M\,|\vol{E_t}-\vol{E}|\big)
\\
&< J_{C_i}(\la),
\end{align*}
for $i$ large enough, using \eqref{eq:almcons} and $\Lip(f_i^{-1})\to 1$. This contradiction shows
\[
J_C(\la)\ge \limsup_{i\to\infty} J_{C_i}(\la),
\]
and hence $J_C(\la)=\lim_{i\to\infty} J_{C_i}(\la)$.
\end{proof}

Theorem~\ref{thm: is cont} allows us to extend properties of the isoperimetric profile for convex bodies with smooth boundary to arbitrary convex bodies. The following result was first proven by E.~Milman

\begin{corollary}[{\cite[Corollary~6.11]{MR2507637}}]
\label{cor:concavidad de perfil}
Let $C\subset\rr^{n+1}$ be a convex body. Then $y_C$ is a concave function. As a consequence, the functions $Y_C$,  $I_C$ and $J_C$ are concave.
\end{corollary}

\begin{proof}
Let $\{C_i\}_{i\in\nn}$ be a sequence of convex bodies with smooth boundaries that converges to $C$ in Hausdorff distance. The functions  $y_{C_i}$ are concave by the results of Kuwert \cite{MR2008339}, see also \cite[Remark 3.3]{bay-rosal}. By Theorem~\ref{thm: is cont}, $y_{C_i}\to y_C$ pointwise in $[0,1]$ and so $y_C$ is concave. Since $Y_C$ is the composition of $y_C$ with an affine function, we conclude that $Y_C$ is also concave. As the composition of a concave function with an increasing  concave function is concave, it follows that $I_C=Y_C^{n/(n+1)}$, $J_C=y_C^{n/(n+1)}$ are concave as well.
\end{proof}

\begin{remark}
The concavity of the isoperimetric profile of an Euclidean convex body with $C^{2,\alpha}$ boundary was proven by Sternberg and Zumbrum \cite{MR1674097}, see also \cite{bay-rosal}. Kuwert later extended this result by showing the concavity of $I_C^{(n+1)/n}$ for convex sets with $C^2$ boundary.
\end{remark}

\begin{corollary}
\label{cor:J_{C_i}toJ_C_unif}
Let $\{C_i\}_{i\in\nn}$ be a sequence of convex bodies in $\rr^{n+1}$ that converges to a convex body $C\subset \rr^{n+1}$ in the Hausdorff topology. Then $J_{C_i}$ $($resp.~$y_{C_i})$ converges to $J_C$ $($resp.~$y_C)$~uniformly on compact subsets of $(0,1)$.
\end{corollary}
\begin{proof}
By Theorem~\ref{thm: is cont} we have that $J_{C_i}\to J_C$ pointwise. By \cite[Thm.~10.8]{roc}, this convergence is uniform on compact sets of $(0,1)$.
\end{proof}

\begin{corollary}
\label{cor:I_CivitoI Cv}
Let $\{C_i\}_{i\in\nn}$ be a sequence of convex bodies in $\rr^{n+1}$ that converges to a convex body $C$ in the Hausdorff topology. Let $v_i\in [0,\vol{C_i}]$, $v\in [0,\vol{C}]$ so that $v_i\to v$. Then $I_{C_i}(v_i)\to I_C(v)$.
\end{corollary}

\begin{proof}
First we consider the case $v=0$. For $i$ sufficiently large, consider Euclidean geodesic balls $B_i \subset \intt(C_i)$ of volume $v_i$. Letting $v_i\to 0$ and taking into account that $I_C(0)=0$, we are done. The case $v=\vol{C}$ is handled taking the complements $C\setminus B_i$ of the balls.

Now assume that $0<v<\vol{C}$. Let $w_i=v_i/\vol{C_i}$ and $w=v/\vol{C}$. Then by the continuity of the volume with respect to the Hausdorff distance \cite[Thm.~1.8.16]{sch} we get $w_i\to w$. Take $\eps>0$ such that $[w-\eps,w+\eps]\subset (0,1)$. For large $i$ we have
\begin{align*}
|J_{C_i}(w_i)-J_C(w)|&\le |J_{C_i}(w_i)-J_{C}(w_i)|+|J_C(w_i)-J_C(w)|
\\
&\le \sup_{x\in[w-\eps,w+\eps]}|J_{C_i}(x)-J_C(x)|+|J_C(w_i)-J_C(w)|.
\end{align*}
By Corollary \ref{cor:J_{C_i}toJ_C_unif}, $J_{C_i}$ converges to $J_C$ uniformly on $[w-\eps,w+\eps]$ and, as $J_C$ is continuous \cite{gallot}, we get  $J_{C_i}(w_i)\to J_C(w)$. From the definition of $J$, $w_i$, and $w$ the proof follows.
\end{proof}

\begin{theorem}
\label{thm:connectedness}
Let $C\subset \rr^{n+1}$  be a convex body, and $E\subset C$ an isoperimetric region. Then $E$ and $C\setminus E$ are connected.
\end{theorem}

\begin{proof}
We shall prove that the function $I_C$ satisfies
\begin{equation}
\label{eq:I_C-strict}
I_C(v_1+v_2)<I_C(v_1)+I_C(v_2),
\end{equation}
whenever $v_1$, $v_2>0$. To prove \eqref{eq:I_C-strict} we shall use the concavity of  $Y_C$ showed in Corollary~\ref{cor:concavidad de perfil} and the fact that $Y_C(0)=0$ to obtain
\[
\frac{Y_C(v_1+v_2)}{v_1+v_2}\le\min\bigg\{\frac{Y_C(v_1)}{v_1},\frac{Y_C(v_2)}{v_2}\bigg\},
\]
what implies
\[
Y_C(v_1+v_2)\le Y_C(v_1)+Y_C(v_2),
\]
as in \cite[Lemma~B.1.4]{bayle}. Raising to the power $n/(n+1)$ we get
\[
I_C(v_1+v_2)\le (I_C(v_1)^{(n+1)/n}+I_C(v_2)^{(n+1)/n})^{n/(n+1)}<I_C(v_1)+I_C(v_1),
\]
where the last inequality follows from $(a+b)^q<a^q+b^q$, for $a$, $b>0$, $q\in (0,1)$, cf. \cite[(2.12.2)]{hlp}. This proves \eqref{eq:I_C-strict}.

If $E\subset C$ were a disconnected isoperimetric region, then $E=E_1\cup E_2$, with $\vol{E}=\vol{E_1}+\vol{E_2}$, and $P_C(E)=P_C(E_1)+P_C(E_2)$, and we should have
\[
I_C(v)=P_C(E)=P_C(E_1)+P_C(E_2)\ge I_C(v_1)+I_C(v_2),
\]
which is a contradiction to \eqref{eq:I_C-strict}. If $E\subset C$ is an isoperimetric region, then $C\setminus E$ is an isoperimetric region and so connected as well.
\end{proof}

In case the boundary of $C$ is of class $C^{2,\alpha}$, Sternberg and Zumbrun \cite{MR1674097} obtained a expression for the second derivative of the perimeter with respect to the volume in formula (2.31) inside Theorem~2.5 of \cite{MR1674097}. Using this formula they obtained in their Theorem~2.6 that a local minimizer $E$ of perimeter (in a $L^1$ sense) has the property that the closure of $\ptl E\cap\intt(C)$ is either connected or it consists of a union of parallel planar (totally geodesic) components meeting $\ptl C$ orthogonally with that part of $C$ lying between any two such totally geodesic components consisting of a cylinder. If $E$ is an isoperimetric region so that the closure of $\ptl E\cap\intt(C)$ consists on more than one totally geodesic component, then Theorem~2.6 in \cite{MR1674097} implies that either $E$ or its complement in $C$ is disconnected, a contradiction to Theorem~\ref{thm:connectedness}. So we have proven

\begin{theorem}
\label{thm:sz-improved}
Let $C$ be a convex body with $C^{2,\alpha}$ boundary, and $E\subset C$ an isoperimetric region. Then the closure of $\ptl E\cap\intt(C)$ is connected.
\end{theorem}

From the concavity of $I_C$ the following properties of the isoperimetric profile of $I_C$ follow. Similar properties can be found in \cite{MR875084}, \cite{MR1161609},  \cite{MR1857855}, \cite{MR2051615} and \cite{MR1803220}.

\begin{proposition}
\label{prp:anal prop isop prof}
Let $C\subset\rr^{n+1}$ be a convex body. Then
\begin{enum}
\item $I_C$ can be extended continuously to $[0,\vol{C}]$ so that $I_C(0)=I_C(\vol{C})=0$.
\item $I_C:[0,\vol{C}]\to\rr^+$ is a positive concave function, symmetric with respect to $\vol{C}/2$, increasing up to $\vol{C}/2$ and decreasing from $\vol{C}/2$. Left and right derivatives $(I_C)'_-(v)$, $(I_C)'_+(v)$, exist for every $v\in (0,\vol{C})$. Moreover, $I_C$ is differentiable $\hh^1$-almost everywhere and we have
\[
I_C(v)=\int_0^v (I_C)'_-(w)\,dw=\int_0^v (I_C)'_+(w)\,dw=\int_0^v I_C'(w)\,dw,
\]
for every $v\in [0,\vol{C}]$.
\item If $E\subset C$ is an isoperimetric region of volume $v\in (0,\vol{C})$, and $H$ is the $($constant$)$ mean curvature of the regular part of $\ptl E\cap\intt(C)$, then
\[
(I_C)'_+(v)\le H\le (I_C)'_-(v).
\]
In particular, if $I_C$ is differentiable at $v$, then the mean curvature of every isoperimetric region of volume $v$ equals $I_C'(v)$.
\end{enum}
\end{proposition}

\begin{proof}
By Theorem~\ref{thm: is cont} we have that $I_C$ is a symmetric, positive, concave function, increasing up to the midpoint and then decreasing. By \cite[Thm.~24.1]{roc}, side derivatives exist for all volumes. By \cite[Thm.~25.3]{roc} differentiability almost everywhere, and absolute continuity \cite[Cor.~24.2.1]{roc} hold, from where the proof of (i) follows.

To prove (ii), take an isoperimetric region $E\subset C$ of volume $v$ and constant mean curvature $H$. By the regularity lemma~\ref{lem:n-7} we can find an open  subset $U$ contained in the regular part of $\ptl E$. Take a nontrivial $C^1$ function $u\ge 0$ with compact support in $U$ that produces an inward normal variation $\{\phi_t\}$ for $t$ small. By the first variation of volume and perimeter we get
\[
\frac{d}{dt}\Big|_{t=0}\vol{\phi_t(E)}=-\int_{\ptl E}u, \qquad \frac{d}{dt}\Big|_{t=0}\pp_C(\phi_t(E))=-\int_{\ptl E}Hu.
\]
So we get $\vol{\phi_t(E)}<\vol{E}$ for $t>0$ and $\vol{\phi_t(E)}>\vol{E}$ for $t<0$. As $\pp_C(\phi_t(E))\le I_C(\vol{\phi_t(E)}$, we have
\[
(I_C)_{-} '(v)= \lim_{\la\uparrow 0}\frac{I_C(v+\la)-I_C(v)}{\la}\ge\frac{d\pp_C(\phi_t(E))}{d\vol{\phi_t(E)}}=H.
\]
Similarly replacing $u$ by $-u$ we get $\la>0$ we find.
\[
(I_C)_{+} '(v)= \lim_{\la\downarrow 0}\frac{I_C(v+\la)-I_C(v)}{\la}\le\frac{d\pp_C(\phi_t(E))}{d\vol{\phi_t(E)}}=H
\]

\end{proof}

Finally, we shall prove in Theorem~\ref{thm:bo-sp} that convex bodies with uniform quotient circumradius/inradius satisfy a uniform relative isoperimetric inequality invariant by scaling. A similar result was proven by Bokowski and Sperner \cite[Satz~3]{bo-sp} using a map different from \eqref{eq:fi}. A consequence of Theorem~\ref{thm:bo-sp} is the existence of a uniform Poincar\'e inequality for balls of small radii inside convex bodies that will be proven in Theorem~\ref{thm:isnqgdbl} and used in the next Section. First we prove the following Lemma.

\begin{lemma}
\label{lem:I_C(v)> cv}
Let $C\subset \rr^{n+1}$ be a convex body and $0<v_0<\vol{C}$. We have
\begin{equation}
\label{eq:I_C(v)> cv.}
I_C(v)\ge \frac{I_C(v_0)}{v_0^{n/(n+1)}}\,v^{n/(n+1)},
\end{equation}
for all $0\le v\le v_0$. As a consequence, we get
\begin{equation}
\label{eq:isopmin}
I_C(v)\ge \frac{I_C(\vol{C}/2)}{(\vol{C}/2)^{n/(n+1)}}\,\min\{v,\vol{C}-v\}^{n/(n+1)},
\end{equation}
for all $0\le v\le\vol{C}$.
\end{lemma}
\begin{proof}
Since  $Y_C=I_C^{(n+1)/n}$ is concave and $Y_C(0)=0$ we get
\[
\frac{Y_C(v)}{v}\ge \frac{Y_C(v_0)}{v_0},
\]
for $0<v\le v_0$. Raising to the power $n/(n+1)$ we obtain \eqref{eq:I_C(v)> cv.}. If $0\le v\le\vol{C}/2$ then \eqref{eq:isopmin} is simply \eqref{eq:I_C(v)> cv.}. If $\vol{C}/2\le v\le\vol{C}$, then $0\le\vol{C}-v\le\vol{C}/2$, we apply \eqref{eq:I_C(v)> cv.} to $\vol{C}-v$ with $v_0=\vol{C}/2$ and we take into account that $I_C(v)=I_C(\vol{C}-v)$ to prove \eqref{eq:isopmin}.
\end{proof}

\begin{remark}
If a set $E$ is isoperimetric in $C$ of volume $\vol{C}/2$, then $\la E$ is isoperimetric in $\la C$ with volume $\vol{\la C}/2$ and perimeter $P_{\la C}(\la E)=\la^n P_C(E)$. So the constant in \eqref{eq:isopmin} satisfies
\[
M_C=\frac{I_{C}(\vol{C}/2)}{(\vol{C}/2)^{n/(n+1)}}=\frac{I_{\la C}(\vol{\la C}/2)}{(\vol{\la C}/2)^{n/(n+1)}},
\]
for any $\la >0$. Hence all dilated convex sets $\la C$, with $\la>0$, satisfy the same isoperimetric inequality
\[
I_{\la C}(v)\ge M_C\,\min\{v,\vol{\la C}-v\}^{n/(n+1)},
\]
for $0<v<\vol{\la C}$.
\end{remark}

\begin{theorem}
\label{thm:bo-sp}
Let $C\subset\rr^{n+1}$ be a convex body, $x,y\in C$, $0<r<R$, such that $\clb(y,r)\subset C \subset \clb(x,R)$. Then there exists a constant $M>0$, only depending on $R/r$ and $n$, such that
\begin{equation}
\label{eq:bokowski}
I_C(v)\ge M \, {\min \{v,\vol{C}-v\}}^{n/(n+1)},
\end{equation}
for all $0\le v\le\vol{C}$.
\end{theorem}

\begin{proof}
Since $\clb(y,r)\subset C\subset \clb(x,R)$ we can construct a bilipschitz map $f:C\to \clb(x,R)$ as in \eqref{eq:f}. Take $0<v<\vol{C}$. By Lemma~\ref{lem:fundamental}, there exists an isoperimetric set $E\subset C$ of volume $v$. By Lemma~\ref{lem:bilip} we have
\begin{align*}
I_C(v)=\pp_C(E)&\ge {(\Lip f)}^{-n} \pp_{ B(x,R)}(f(E)),
\\
\vol{\clb(x,R)\setminus f(E)}&\ge (\Lip f^{-1})^{-(n+1)}\big( \vol{C\setminus E} \big),
\\
\vol{f(E)}&\ge (\Lip(f^{-1})^{-(n+1)}\vol{E}.
\end{align*}

We know \cite[Cor.~1.29]{gi} that for $f(E)\subset\clb(x,R)$ we have the isoperimetric inequality
\[
\pp_{\clb(x,R)}(f(E))\ge M(n)\,\min\{\vol{f(E)},\vol{\clb(x,R)}-\vol{f(E)}\}^{n/(n+1)},
\]
where $M(n)$ is a constant that only depends on the dimension $n$. So we get
\[
I_C(v)\ge M(n)\,\big((\Lip f)(\Lip f^{-1})\big)^{-n}\min\{v,\vol{C}-v\}^{n/(n+1)}.
\]
As $\clb(x,R)\subset\clb(y,2R)$, Corollary~\ref{cor:estilipconst} provides upper bounds of $\Lip(f)$,  $\Lip(f^{-1})$ only depending on $R/r$. This completes the proof of the Proposition.
\end{proof}

\begin{theorem}
\label{thm:isnqgdbl}
Let $C\subset \rr^{n+1}$ a convex body. Given $r_0>0$, there exist positive constants $M$, $\ell_1$, only depending on $r_0$ and $C$, and a universal positive constant $\ell_2$ so that
\begin{equation}
\label{eq:isnqgdbl1}
I_{\clb_C(x,r)}(v)\ge M\, {\min \{v,\vol{\clb_C(x,r)}-v\}}^{n/(n+1)},\end{equation}
for all $x\in C$, $0<r\le r_0$, and $0<v<\vol{\clb_C(x,r)}$. Moreover
\begin{equation}
\label{eq:isnqgdbl1a}
\ell_1 r^{n+1} \le \vol{\clb_C(x,r)} \le \ell_2 r^{n+1},
\end{equation}
for any $x\in C$, $0<r\le r_0$.
\end{theorem}

\begin{proof}
To prove \eqref{eq:isnqgdbl1} we only need an upper estimate of the quotient of $r$ over the inradius of $\clb(x,r)$ by Theorem~\ref{thm:bo-sp}. By the compactness of $C$ we deduce that
\begin{equation}
\label{eq:claim inf x in inr Bx}
\inf_{x\in C}\inr(\clb_C(x,r_0))>0
\end{equation}
Hence, for every $x\in C$, we always can find a point $y(x)\in \clb_C(x,r_0)$ and a positive constant $\de>0$ independent of $x$ such that,
\begin{equation}
\label{eq:B_{y(x)}(de)}
\clb(y(x),\de)\subset \clb_C(x,r_0)\subset\clb(x,r_0).
\end{equation}
Now take $0<r\le r_0$. Let $0<\la\le 1$ so that $r=\la r_0$, and denote by $h_{x,\la}$ the homothety of center $x$ and radius $\la$. Then we have $h_{x,\la}(\clb(y(x),\de))\subset h_{x,\la}(\clb_C(x,r_0))$ and so
\begin{equation*}
\clb(h_{x,\la}(y(x)),\la\de)\subset\clb_{h_{x,\la}(C)}(x,\la r_0)\subset \clb_C(x,\la r_0),
\end{equation*}
since $h_{x,\la}(C)\subset C$ as $x\in C$, $0<\la\le 1$, and $C$ is convex. Again by Theorem~\ref{thm:bo-sp}, a relative isoperimetric inequality is satisfied in $\clb_C(x,r)$ with a constant $M$ that only depends on $r_0/\de$.

We now prove \eqref{eq:isnqgdbl1a}. Since $\vol{\clb_C(x,r)}\le \vol {\clb(x,r)}$, it is enough to take $\ell_2=\omega_{n+1}=\vol{\clb(0,1)}$. For the remaining inequality, using the same notation as above, we have
\begin{align*}
\vol{\clb(x,r)\cap C}&= \vol{\clb(x,\la r_0)\cap C}\ge \vol{h_{x,\la}(\clb(x,r_0) \cap C)}
\\
&=\la^{n+1} \vol{\clb(x,r_0)\cap C}\ge \la^{n+1} \vol{\clb(y(x),\de)}
\\
&=\omega_{n+1}(\de/r_0)^{n+1}\,r^{n+1},
\end{align*}
and we take $\ell_1=\omega_{n+1}(\de/r_0)^{n+1}$.
\end{proof}

\section{Convergence of isoperimetric regions}

Let $\{C_i\}_{i\in\nn}$ be a sequence of convex bodies converging in Hausdorff distance to a convex body $C$, and $\{E_i\}_{i\in\nn}$ a sequence of isoperimetric regions in $C_i$ of volumes $v_i$ weakly converging to some isoperimetric region $E\subset$ of volume $v=\lim_{i\to\infty} v_i$. The main result in this Section is that $E_i$ converges to $E$ in Hausdorff distance, and also their relative boundaries. As a byproduct, we shall also prove that there exists always in $C$ an isoperimetric region with connected boundary. It is still an open question to show that \emph{every} isoperimetric region on a convex body has connected boundary.

We prove first a finite number of Lemmata

\begin{lemma}
\label{lem:link I laC I C}
Let $C$ be a convex body, and $\la> 0$. Then
\begin{equation}
\label{eq:proflac}
I_{\lambda C}(\la^{n+1}v)={\lambda}^nI_{C}(v),
\end{equation}
for all $0\le v\le\min\{\vol{C},\vol{\la C}\}$.
\end{lemma}

\begin{proof}
For $v$ in the above conditions we get
\begin{align*}
I_{\lambda C}({\lambda}^{n+1}v)&=
\inf \big\{ P_{\la C}(\la E) : \la E \subset \lambda C,\ \vol{\la E} ={\lambda}^{n+1} v  \big\}
\\
&= \inf \Big \{\la^n P_C(E) : E\subset C,\ \vol{E} = v \Big \}
\\
& =\lambda^{n}I_C(v).
\end{align*}
\end{proof}

\begin{remark}
Lemma~\ref{lem:link I laC I C} implies
\begin{equation}
\label{eq:Y laC=...}
Y_{\la C}(\la^{n+1}v)=\la^{n+1} Y_{ C}(v)
\end{equation}
for any ${\la}>0$ and $0\le v\le\min\{\vol{C},\vol{\la C}\}$.
\end{remark}

\begin{lemma}
\label{lem:_IC<la C}
Let $C$ be a  convex body, $\la\ge 1$. Then
\begin{equation}
\label{eq:ilacic}
I_{\la C}(v)\ge I_C(v)
\end{equation}
for all $0\le v\le\vol{C}$.
\end{lemma}

\begin{proof}
Let $Y_{\la C}=I_{\la C}^{(n+1)/n}$. We know from Corollary~\ref{cor:concavidad de perfil} that $Y_C$ is a concave function with $Y_{\la C}(0)=0$. Since $\la \ge 1$, for $v>0$ we have
\[
\frac{Y_{\la C}(v)}{v}\ge\frac{Y_{\la C}(\la^{n+1} v)}{\la^{n+1} v},
\]
what implies, using \eqref{eq:Y laC=...},
\[
\la^{n+1} Y_{\la C}(v)\ge Y_{\la C}(\la ^{n+1} v)=\la^{n+1}Y_C(v).
\]
This proves \eqref{eq:ilacic}.
\end{proof}

In a similar way to \cite[p.~18]{le-ri}, given a convex body $C$ and $E\subset C$, we define a function $h:C\times(0,+\infty)\to (0,\tfrac{1}{2})$ by
\begin{equation}
\label{eq:defh}
h(E,C,x,R)=\frac{\min\big\{ \vol{E\cap B_C(x,R)},\vol{B_C(x,R)\setminus E} \big\}}{\vol{B_C(x,R)}},
\end{equation}
for $x\in C$ and $R>0$. When $E$ and $C$ are fixed, we shall simply denote
\begin{equation}
h(x,R)=h(E,C,x,R).
\end{equation}
\begin{lemma}
\label{lem:fv}
For any $v>0$, consider the function $f_v:[0,v]\to\rr$ defined by
\begin{equation*}
f_v(s)=s^{-n/(n+1)}\,\bigg(\bigg(\frac{v-s}{v}\bigg)^{n/(n+1)}-1\bigg).
\end{equation*}
Then there is a constant $0<c_2<1$ that does not depends on $v$ so that $f_v(s)\ge -(1/2)\,v^{-n/(n+1)}$ for all $0\le s\le c_2\,v$.
\end{lemma}

\begin{proof}
By continuity, $f_v(0)=0$. Observe that $f_v(v)=-v^{-n/(n+1)}$ and that, for $s\in [0,1]$, we have $f_v(sv)=f_1(s)\,v^{-n/(n+1)}$. The derivative of $f_1$ in the interval $(0,1)$ is given by 
\[
f_1'(s)=\frac{n}{n+1}\,\frac{(s-1)+(1-s)^{n/(n+1)}}{s-1}\,s^{-1-n/(n+1)},
\]
which is strictly negative and so $f_1$ is strictly decreasing. Hence there exists $0<c_2<1$ such that $f_1(s)\ge -1/2$ for all $s\in [0,c_2]$. This implies $f_v(s)=f_1(s/v)\,v^{-n/(n+1)}\ge -(1/2)\,v^{-n/(n+1)}$ for all $s\in [0,c_2 v]$.
\end{proof}

Now we prove a key density result for isoperimetric regions. Its proof is inspired by  Lemma~4.2 of the paper by Leonardi and Rigot \cite{le-ri}. Similar results for quasi-minimizing sets were previously proven by David and Semmes \cite{MR1625982}.

\begin{theorem}
\label{thm:leon rigot lem 42}
\mbox{}
Let $C\subset \rr^{n+1}$ be a convex body, and $E\subset C$ an isoperimetric region of volume $0<v<|C|$.  Choose $\eps$ so that
\begin{equation}
\label{eq:epsfine}
0<\eps<\min\bigg\{\frac{v}{\ell_2}, \frac{\vol{C}-v}{\ell_2},c_2v,c_2(\vol{C}-v),\frac{I_C(v)^{n+1}}{\ell_28^{n+1}v^{n}},\frac{I_C(v)^{n+1}}{\ell_28^{n+1}(\vol{C}-v)^{n}}\bigg\},
\end{equation}
where $c_2$ is the constant in Lemma~\ref{lem:fv}.

Then, for any $x\in C$ and $R\le 1$ so that $h(x,R)\le\eps$, we get
\begin{equation}
\label{eq:hr/2=0}
h(x,R/2)=0.
\end{equation}
Moreover, in case $h(x,R)=\vol{E\cap B_C(x,R)}\vol{B_C(x,R)}^{-1}$, we get $|E\cap B_C(x,R/2)|=0$ and, in case $h(x,R)=\vol{B_C(x,R)\setminus E}\vol{B_C(x,R)}^{-1}$, we have $|B_C(x,R/2)\setminus E|=0$.
\end{theorem}

\begin{proof}
From Lemma~\ref{lem:I_C(v)> cv} we get
\begin{equation}
\label{eq:c1}
I_C(w)\ge c_1w^{n/(n+1)},\qquad\text{where}\qquad c_1=v^{-n/(n+1)}I_C(v),
\end{equation}
for all $0\le w\le v$.

Assume first that
\[
h(x,R)=\frac{\vol{E\cap B_C(x,R)}}{\vol{B_C(x,R)}}.
\]
Define $m(t)=\vol{E\cap B_C(x,t)}, 0<t\le R$. Thus $m(t)$ is a non-decreasing function. For $t\le R\le 1$ we get
\begin{equation}
\label{eq:m(t)<eps}
m(t)\le m(R)=\vol{E\cap B_C(x,R)}=h(x,R) \,\vol{B_C(x,R)}\le h(x,R)\,\ell_2 R^{n+1}\le\eps\ell_2<v,
\end{equation}
by \eqref{eq:epsfine}. So we obtain $(v-m(t))>0$.

By the coarea formula, when $m'(t)$ exists, we get
\begin{equation}
\label{eq:le ri cofo}
m'(t)=\frac{d}{dt}\int_0^t\h^{n}(E\cap \ptl B_C(x,s))ds=\h^{n}(E\cap \ptl B_C(x,t)),
\end{equation}
where we have denoted $\ptl B_C(x,t)=\ptl B(x,t)\cap\intt(C)$. Define
\begin{equation}
\label{eq:def la, E(t)}
\la(t)=\frac{v^{1/{(n+1)}}}{(v-m(t))^{1/{(n+1)}}},\qquad E(t)=\la(t)(E\setminus B_C(x,t)).
\end{equation}
Then $E(t)\subset \la(t)C$ and $\vol{E(t)}=\vol{E}=v$. 
By Lemma \ref{lem:_IC<la C}, we get $I_{\la(t) C}\ge I_C$ 
since $\la(t)\ge 1$. Combining this with \cite[Cor.~5.5.3]{Ziemer}, equation~\eqref{eq:le ri cofo}, and elementary properties of the perimeter functional, we get
\begin{equation}
\label{eq:diffineq}
\begin{split}
I_C(v)&\le I_{\la(t) C}(v)\le \pp_{\la(t)C}(E(t))=\la^n(t)\,\pp_C(E\setminus B_C(x,t))
\\
&\le \la^n(t)\big(\pp_C(E)-\pp(E,B_C(x,t))+\h^{n}(E\cap \ptl B_C(x,t))\big)
\\
&\le \la^n(t)\big(\pp_C(E)-\pp_C(E\cap B_C(x,t))+2\h^{n}(E\cap \ptl B_C(x,t))\big)
\\
&\le \la^n(t)\big(I_C(v)-c_1m(t)^{n/{(n+1)}}+2m'(t)\big),
\end{split}
\end{equation}
where $c_1$ is the constant in \eqref{eq:c1}. Multiplying both sides by $I_C(v)^{-1}{\la(t)}^{-n}$ we find
\begin{equation}
\label{eq:flex..}
{\la(t)}^{-n}-1+\frac{c_1}{I_C(v)}m(t)^{n/{(n+1)}}\le \frac{2}{I_C(v)}m'(t).
\end{equation}
Set
\begin{equation}
\label{eq:def of a, b}
a=\frac{2}{I_C(v)},\qquad b=\frac{c_1}{I_C(v)}=\frac{1}{v^{n/(n+1)}}.
\end{equation}
From the definition \eqref{eq:def la, E(t)} of $\la(t)$ we get
\begin{equation}
\label{eq:f(m(t))le am'(t)}
f(m(t))\le am'(t)\,\qquad\h^1\text{-a.e},
\end{equation}
where
\begin{equation}
\frac{f(s)}{s^{n/(n+1)}}=b+\frac{\big(\frac{v-s}{v}\big)^{n/{(n+1)}}-1}{s^{n/(n+1)}}.
\end{equation}
By Lemma~\ref{lem:fv}, there exists a universal constant $0<c_2<1$, not depending on $v$, so that
\begin{equation}
\label{eq:epsilon le-ri}
\frac{f(s)}{s^{n/{n+1}}}\ge b/2\qquad \text{whenever}\qquad 0< s\le c_2 v.
\end{equation}
Since $\eps\le c_2 v$ by \eqref{eq:epsfine}, equation \eqref{eq:epsilon le-ri} holds in the interval $[0,\eps]$. If there were $t\in[R/2,R]$ such that $m(t)=0$ then, by monotonicity of $m(t)$, we would conclude $m(R/2)=0$ as well. So we assume $m(t)>0$ in $[R/2,R]$. Then by
\eqref{eq:f(m(t))le am'(t)} and \eqref{eq:epsilon le-ri}, we get
\[
b/2a\le \frac{m'(t)}{m(t)^{n/{n+1}}},\,\qquad \h^1\text{-a.e.}
\]
Integrating between $R/2$ and $R$ we get by \eqref{eq:m(t)<eps}
\[
bR/4a\le(m(R)^{1/{(n+1)}}-m(R/2)^{1/{(n+1)}})\le m(R)^{1/{(n+1)}}\le (\eps\ell_2)^{1/(n+1)}R\le(\eps\ell_2)^{1/(n+1)}.
\]
This is a contradiction, since $\eps\ell_2<(b/4a)^{n+1}=I_C(v)^{n+1}/(8^{n+1} v^n)$ by \eqref{eq:epsfine}. So the proof in case $h(x,R)=\vol{E\cap B_C(x,R)}\,(\vol{B_C(x,R))}^{-1}$ is completed.

For the remaining case, when $h(x,R)=\vol{B_C(x,R)}^{-1}\vol{B_C(x,R)\setminus E}$, we replace $E$ by $C\setminus E$, which is also an isoperimetric region, and we are reduced to the previous case.
\end{proof}

\begin{remark}
Case $h(x,R)=\vol{B_C(x,R)}^{-1}\vol{B_C(x,R)\setminus E}$ is treated in \cite{le-ri} in a completely different way using the monotonicity of the isoperimetric profile in Carnot groups.
\end{remark}

We define the sets
\begin{align*}
E_1&=\{x\in C:\exists\, r>0\ \text{such that}\ \vol{B_C(x,r)\setminus E}=0\},
\\
E_0&=\{x\in C:\exists\, r>0\ \text{such that}\ \vol{B_C(x,r)\cap E}=0\},
\\
S&=\{x\in C: h(x,r)>\eps\ \text{for all}\ r\le 1\}.
\end{align*}

In the same way as in Theorem~4.3 of \cite{le-ri} we get

\begin{proposition}
Let $\eps$ be as in Theorem~\ref{thm:leon rigot lem 42}. Then we have
\begin{enumerate}
\item[(i)] $E_0$, $E_1$ and $S$ form a partition of $C$.
\item[(ii)] $E_0$ and $E_1$ are open in $C$.
\item[(iii)] $E_0=E(0)$ and $E_1=E(1)$.
\item[(iv)] $S=\ptl E_0=\ptl E_1$, where the boundary is taken relative to $C$.
\end{enumerate}
\end{proposition}

As a consequence we get the following two corollaries

\begin{corollary}[Lower density bound]
\label{cor:monotonicity}
Let $C\subset \rr^{n+1}$ be a convex body, and $E\subset C$ an isoperimetric region of volume $v$. Then there exists a constant $M>0$, only depending on $\eps$, on Poincar\'e constant for $r\le 1$, and on an Ahlfors constant $\ell_1$, such that
\begin{equation}
\label{eq:monotonicity}
\pp(E,B_C(x,r))\ge M r^n,
\end{equation}
for all $x\in\ptl E_1$ and $r\le 1$.
\end{corollary}

\begin{proof}
If $x\in \ptl E_1$, the choice of $\eps$ and 
the relative isoperimetric inequality \eqref{eq:isnqgdbl1} 
give
\begin{equation*}
\begin{split}
\pp(E,B_C(x,r))&\ge M\min\{\vol{E\cap B_C(x,r)}, \vol{B_C(x,r) \setminus E}\}^{n/{(n+1)}}
\\
&=M\,(\vol{B_C(x,r)}\,h(x,r))^{n/(n+1)}\ge M(\vol{B_C(x,r)}\,\eps)^{n/(n+1)}
\\
&\ge M\,(\ell_1\eps)^{n/(n+1)}\,r^n.
\end{split}
\end{equation*}
This implies the desired inequality.
\end{proof}

\begin{remark}
\label{rem:uniformeps}
If $C_i$ is a sequence of convex bodies converging to a convex body $C$ in Hausdorff distance, and $E_i\subset C_i$ is a sequence of isoperimetric regions converging weakly to an isoperimetric region $E\subset C$ of volume $0<v<\vol{C}$, then a constant $M>0$ in \eqref{eq:monotonicity} can be chosen independently of $i\in\nn$. In fact, by \eqref{eq:epsfine}, the constant $\eps$ only depends on $\vol{E_i}$, $\vol{C_i}-\vol{E_i}$, and $I_{C_i}(\vol{E_i})$, which are uniformly bounded since $\vol{C_i}\to\vol{C}$ and $\vol{E_i}\to\vol{E}$. By the convergence in Hausdorff distance of $C_i$ to $C$, both a lower Ahlfors constant $\ell_1$ and a Poincar\'e constant can be chosen uniformly for all $i\in\nn$.
\end{remark}

\begin{remark}
The classical monotonicity formula for rectifiable varifolds \cite{simon} can be applied in the interior of $C$ to get the lower bound \eqref{eq:monotonicity} for small $r$.
Assuming $C^2$ regularity of the boundary of $C$ (convexity is no longer needed), a monotonicity formula for varifolds with free boundary under boundedness condition on the mean curvature have been obtained by Gr\"uter and Jost \cite{MR863638}. This monotonicity formula implies the lower density bound \eqref{eq:monotonicity}. 
\end{remark}

Now we prove that isoperimetric regions also converge in Hausdorff distance to their weak limits, which are also isoperimetric regions. It is necessary to choose a representative of the isoperimetric regions in the class of finite perimeter so that Hausdorff convergence makes sense: we simply consider the closure of the set $E_1$ of points of density one.

\begin{theorem}
\label{thm:EitoE Haus}
Let $\{C_i\}_{i\in\nn}$ be a sequence of convex bodies that converges in Hausdorff distance to a convex body $C$. Let $E_i\subset C_i$ be a sequence of isoperimetric regions of volumes $v_i\to v\in (0,\vol{C})$. Let $f_i:C_i\to C$ be a sequence of bilipschitz maps with $\Lip(f_i), \Lip(f_i^{-1})\to 1$.

Then there is an isoperimetric set $E\subset C$ such that a subsequence of $f_i(E_i)$ converges to $E$ in Hausdorff distance. Moreover, $E_i$ converges to $E$ in Hausdorff distance.
\end{theorem}

\begin{proof}
The sequence $\{f_i(E_i)\}_{i\in\nn}$ has uniformly bounded perimeter and so a subsequence, denoted in the same way, converges in $L^1(C)$ to a finite perimeter set $E$, which has volume $v$. The set $E$ is isoperimetric in $C$ since the sets $E_i$ are isoperimetric in $C_i$ and $I_{C_i}(v_i)\to I_C(v)$ by Corollary~\ref{cor:I_CivitoI Cv}.

By Remark~\ref{rem:uniformeps}, we can choose $\eps>0$ so that Theorem~\ref{thm:leon rigot lem 42} holds with this $\eps$ for all $i\in\nn$. Choosing a smaller $\eps$ if necessary we get that, for any $x\in C$ and $0<r\le 1$, whenever $h(f_i(E_i),C,x,r)\le\eps$, we get $h(f_i(E_i),C,x,r/2)=0$.

We now prove that $f_i(E_i)\to E$ in Hausdorff distance. As  $\chi_{f_i(E_i)}\to \chi_E$ in $L^1(C)$, we can choose a sequence $r_i\to 0$ so that 
\begin{equation}
\label{eq:EitoE Haus3}
\vol{f_i(E_i)\,\triangle\, E}<r_i^{n+2}.
\end{equation}
Now fix some $0<r<1$ and assume that, for some subsequence, there exist $x_i\in f_i(E_i)\setminus E_{r}$, where $E_r=\{x\in C: d(x,E)\le r\}$. Choose $i$ large enough so that $r_i<\min\{\tfrac{\ell_1}{2},r\}$. Then, by \eqref{eq:EitoE Haus3},
\begin{equation}
\label{eq:EitoE Haus4}
|f_i(E_i)\cap B_C(x_i,r_i)|\le |f_i(E_i)\setminus E|\le |f_i(E_i)\triangle E|< r_i^{n+2}<\frac{\ell_1r_i^{n+1}}{2}\le \frac{|B_C(x_i,r_i)|}{2}.
\end{equation}
So, for $i$ large enough, we get
\[
h(f_i(E_i),C,x_i,r_i)=\frac{|f_i(E_i)\cap B_C(x_i,r_i)|}{|B_C(x_i,r_i)|}<\ell_1^{-1}r_i\le \eps.
\]
By Theorem~\ref{thm:leon rigot lem 42}, we conclude that $\vol{f_i(E_i)\cap B_{C}(x,r_i/2)}=0$. The normalization condition imposed on the isoperimetric regions implies a contradiction that shows that $f_i(E_i)\subset (E)_{r}$ for $i$ large enough. In a similar way we get that $E\subset f_i(E_i)_r$, which proves that the Hausdorff distance between $E$ and $f_i(E_i)$ is less than an arbitrary $r>0$. So $f_i(E_i)\to E$ in Hausdorff distance.

Now we prove $\de(E_i,E)\to 0$. By the triangle inequality we have
\[
\de(E_i,E)\le\de(f_i(E_i),E)+\de(f_i(E_i),E_i).
\]
It only remains to show that $\de(f_i(E_i),E_i)\to 0$. For $x\in E_i$ we have
\[
\dist(f_i(x),E_i)\le |f_i(x)-x|.
\]
Assume  that $r>0$ is as in definition \eqref{eq:fi} of $f_i$. Recall that $B(0,2r)\subset C_i\cap C$ and that $C_i\cup C \subset B(0,R)$. Then by \eqref{eq:fi2} we get $|f_i(x)-x|=0$ if $|x|\le r$ and
\[
|f_i(x)-x|\le \frac{(R-r)}{r}\, \bigg|\, \rho_i\big(\frac{x}{|x|}\big)-\rho\big(\frac{x}{|x|}\big)\, \bigg|
\]
if  $|x|\ge r$. Lemma~\ref{lem:radialconv} then implies the existence of a sequence of positive real numbers $\eps_i\to 0$ such that $|f_i(x)-x|\le \eps_i$ for all $x\in E_i$. We conclude
that
\begin{equation*}
\label{eq:EitoEHaus.1}
f_i(E_i)\subset (E_i)_{\eps_i}.
\end{equation*}
Writing $E_i=f_i^{-1}(f_i(E_i))$ and reasoning as above with $f_i^{-1}$ instead of $f_i$ we obtain 
\begin{equation*}
\label{eq:EitoEHaus.2}
E_i\subset (f_i(E_i))_{\eps_i},
\end{equation*}
By the definition of the Hausdorff distance $\de$, we get $\de(f_i(E_i),E_i)\to 0$.
\end{proof}

Recall that in Theorem~\ref{thm:sz-improved} we showed that the boundaries of isoperimetric regions in convex sets with $C^{2,\alpha}$ boundary are connected. For arbitrary convex sets we have the following

\begin{theorem}
\label{thm: ptl conect}
Let $C\subset\rr^{n+1}$ be a  convex body. For every volume $0<v<\vol{C}$ there exists an isoperimetric region in $C$ of volume $v$ with connected boundary.
\end{theorem}

We shall use the following result in the proof of Theorem~\ref{thm: ptl conect}.

\begin{theorem}
\label{thm:ptl haus}
Let $\{C_i\}_{i\in\nn}$ a sequence of convex bodies converging in Hausdorff distance to a convex body $C$, and let $E_i\subset C_i$ be a sequence of isoperimetric regions converging in Hausdorff distance to an isoperimetric region $E\subset C$.

Then a subsequence of $\cl{\ptl E_i\cap\intt(C_i)}$ converges to $\cl{\ptl E\cap\intt(C)}$ in Hausdorff distance as well.
\end{theorem}

\begin{proof}[Proof of Theorem~\ref{thm: ptl conect}]
Let $C_i\subset\rr^{n+1}$ be convex bodies with $C^{2,\alpha}$ boundary converging to $C$ in Hausdorff distance. Let $E_i\subset C_i$ be isoperimetric regions of volumes approaching $v$. By Theorem~\ref{thm:EitoE Haus}, a subsequence of the sets $E_i$ converges to $E$ in Hausdorff distance, where $E\subset C$ is an isoperimetric region of volume $v$. By Theorem~\ref{thm:ptl haus}, a subsequence of the sets $\cl{\ptl E_i\cap\intt(C_i)}$ converges to $\cl{\ptl E\cap\intt(C)}$ in Hausdorff distance. Theorem~\ref{thm:sz-improved} implies that the sets $\cl{\ptl E_i\cap\intt(C_i)}$ are connected. By Proposition A.1.7 in \cite{kr-pa}, $\cl{\ptl E\cap\intt(C)}$ is connected as well.
\end{proof} 

\begin{proof}[Proof of Theorem~\ref{thm:ptl haus}]
We shall prove that that the sequence $\{\cl{\ptl E_i\cap\intt(C_i)}\}_{i\in\nn}$ converges to $\cl{\ptl E\cap\intt(C)}$ in Kuratowski sense \cite[4.4.13]{amb}
\begin{enumerate}
\item[1.] If $x=\lim_{j\to\infty} x_{i_j}$ for some subsequence $x_{i_j}\in \cl{\ptl E_{i_j}\cap\intt(C_i)}$, then $x\in\cl{\ptl E\cap\intt(C)}$, and
\item[2.] If $x\in \cl{\ptl E\cap\intt(C)}$, then there exists a sequence $x_i\in\cl{\ptl E\cap\intt(C)}$ converging to $x$.
\end{enumerate}
Assume 1 does not hold. To simplify the notation we shall assume that $x=\lim_{i\to\infty} x_i$, with $x_i\in\cl{\ptl E_i\cap \intt(C_i)}$. If $x\not\in\cl{\ptl E\cap\intt(C)}$ we had $x\in \intt(E)\cup\intt(C\setminus E)$. If $x\in\intt(E)$, then there exists $r>0$ such that $\vol {B(x,r)\cap (C\setminus E)}=0$. Since $x_i\to x$, and $E_i$, $C_i$ converge to $E,C$ in Hausdorff sense, respectively, we conclude by \cite[ Proposition 4.4.14]{amb} that $\clb(x_i,r)\cap (C_i\setminus E_i)\to \clb(x,r)\cap (C\setminus E)$ in the Hausdorff sense as well. Thus by \cite[Lemma III.1.1]{ch} we get
\[
\limsup_{i\to\infty} \vol {B(x_i,r)\cap (C_i\setminus E_i)}\le\vol { B(x,r)\cap (C\setminus E)}=0.
\]
Now if $\eps>0$ is as in Theorem~\ref{thm:leon rigot lem 42}, we get  $\vol {B(x_i,r)\cap (C_i\setminus E_i)}\le\eps$ for all large $i\in \nn$ which implies $\vol {B(x_i,r/2)\cap (C_i\setminus E_i)}=0$. This contradicts the fact that $x_i\in \cl{\ptl E_i\cap\intt(C_i)}$. Assuming $x\in C\setminus E$ and arguing similarly we would find $\vol {B(x_i,r/2)\cap \intt(E_i)}=0$. Thus $x\in \cl{\ptl E\cap\intt(C)}$.

Assume now that $2$ does not hold. Then there exists $x\in\cl{\ptl E\cap\intt(C)}$ so that no sequence in $\cl{\ptl E_i\cap\intt(C_i)}$ converges to $x$. We may assume that, passing to a subsequence if necessary, that there exists $\eta>0$ so that $B_C(x,\eta)$ does not contain any point in $\cl{\ptl E_i\cap\intt(C_i)}$. The radius $\eta$ can be chosen less than $\eps$. Reasoning as in Case 1, we conclude that either $B_C(x,\eta/2)\cap E_i=\emptyset$ or $B_C(x,\eta/2)\cap (C\setminus E_i)=\emptyset$.
\end{proof}

\section{The asymptotic isoperimetric profile of a convex body}

In this section we shall prove that isoperimetric regions of small volume inside a convex body concentrate near boundary points whose tangent cone has the smallest possible solid angle. This will be proven by rescaling the isoperimetric regions and then studying their convergence, as in Morgan and Johnson \cite{MR1803220}. We shall recall first some results on convex cones.

Let $K\subset \rr^{n+1}$ be a closed convex cone with vertex $p$ . Let $\alpha(K)=\hh^n(\ptl{B}(p,1)\cap \intt(K))$ be the \emph{solid angle} of $K$. It is known that the geodesic balls centered at the vertex are isoperimetric regions in $K$, \cite{lions-pacella}, \cite{r-r}, and that they are the only ones \cite{FI} for general convex cones, without any regularity assumption on the boundary. The isoperimetric profile of $K$ is given by
\begin{equation}
\label{eq:isopsolang}
I_K(v)={\alpha(K)}^{1/(n+1)}\,(n+1)^{n/(n+1)}v^{n/(n+1)}.
\end{equation}
Consequently the isoperimetric profile of a convex cone is completely determinated by its solid angle.

We define the tangent cone $C_{p}$ of a convex body $C$ at a given boundary point $p\in\ptl C$ as the closure of the set
\[
\bigcup_{\la > 0} h_{p,\la} (C),
\]
where $ h_{p,\la}$ denotes the dilation of center $p$ and factor $\la$. The solid angle $\alpha(C_p)$ of $C_p$ will be denoted by $\alpha(p)$. Tangent cones to convex bodies have been widely considered in convex geometry under the name of supporting cones \cite[\S~2.2]{sch} or projection cones \cite{MR920366}. In the following result, we prove the lower semicontinuity of the solid angle of tangent cones in convex sets.

\begin{lemma}
\label{lem:losemcontangcon}
Let $C\subset \rr^{n+1}$ be a convex body, $\{p_i\}_{i\in\nn}\subset \ptl C$ so that $p=\lim_{i\to\infty} p_i$. Then
\begin{equation}
\label{eq:liminfalpha}
\alpha(p) \le \liminf_{i\to\infty}\alpha({p_i}).
\end{equation}
In particular, this implies the existence of points in $\ptl C$ whose tangent cones are minima of the solid angle function.
\end{lemma}

\begin{proof}
We may assume that $\alpha({p_i})$ converges to $\liminf_{i\to\infty}\alpha({p_i})$ passing to a subsequence if necessary. Since the sequence $C_{p_i}\cap\clb(p_i,1)$ is bounded for the Hausdorff distance, we can extract a subsequence (denoted in the same way) converging to a convex body $C_\infty\subset\clb(p,1)$. It is easy to check that $C_\infty$ is the intersection of a closed convex cone $K_\infty$ of vertex $p$ with $\clb(p,1)$, and that $C_p\subset K_\infty$. By the continuity of the volume with respect to the Hausdorff distance we have
\[
\alpha(p)=\vol{C_p\cap\clb(p,1)}\le\vol{C_\infty}=\lim_{i\to\infty} \vol{C_{p_i}\cap\clb(p_i,1)}=\lim_{i\to\infty}\alpha({p_i}),
\]
yielding \eqref{eq:liminfalpha}. To prove the existence of tangent cones with the smallest solid angle, we simply take a sequence $\{p_i\}_{i\in\nn}$ of points at the boundary of $C$ so that $\alpha(p_i)$ converges to $\inf\{\alpha(p):p\in\ptl C\}$, we extract a convergent subsequence, and we apply the lower semicontinuity of the solid angle function.

\end{proof}

The isoperimetric profiles of tangent cones which are minima of the solid angle function coincide. The common profile will be denoted by $I_{C_{\min}}$.

\begin{proposition}
\label{prp:ICleICmin}
Let $C\subset\rr^{n+1}$ be a convex body. Then
\begin{equation}
\label{eq:ICleICmin}
I_C(v)\le I_{C_{\min}}(v),
\end{equation}
for all $0\le v\le\vol{C}$.
\end{proposition}

\begin{remark}
\label{rem:half-plane}
A closed half-space $H\subset\rr^{n+1}$ is a convex cone with the largest possible solid angle. Hence, for any convex body $C\subset\rr^{n+1}$, we have
\[
I_C(v)\le I_H(v),
\] 
for all $0\le v\le\vol{C}$.
\end{remark}

\begin{remark}
Proposition~\ref{prp:ICleICmin} gives an alternative proof of the fact that $E\cap\ptl C\neq\emptyset$ when $E\subset C$ is isoperimetric since, in case $E\cap\ptl C$ is empty, then $E$ is an Euclidean ball.
\end{remark}

\begin{proof}[Proof of Proposition~\ref{prp:ICleICmin}]
Fix some volume $0<v<\vol{C}$. Let $p\in\ptl C$ such that $I_{C_p}=I_{C_{\min}}$. Let $r>0$ such that $\vol{B_C(p,r)}=v$. The closure of the set $\ptl B(p,r)\cap \intt(C)$ is a geodesic sphere of the closed cone $K_p$ of vertex $p$ subtended by the closure of $\ptl B(p,r)\cap \intt(C)$. If $S= \ptl B(p,r)\cap \intt(C) $ then $S= \ptl B(p,r)\cap \intt(K_p)$ as well. By the convexity of $C$, $B(p,r)\cap \intt(K_p) \subset B(p,r)\cap \intt(C)$ and so $v_0=H^{n+1}(B(p,r)\cap\intt(K_p))\le v$. Since $K_p\subset C_p$, \eqref{eq:isopsolang} implies $H^n(S)\le I_{C_{\min}}(v_0)$. So we have
\[
I_C(v) \le \pp_C(B_C(p,r))=\h^n(S)\le I_{C_{\min}}(v_0) \le I_{C_{\min}}(v),
\]
as $I_{C_{\min}}$ is an increasing function. This proves \eqref{eq:ICleICmin}.
\end{proof}

We now prove the following result which strongly depends on the paper by Figalli and Indrei \cite{FI}.

\begin{lemma}
\label{lem:mincone}
Let $K\subset\rr^{n+1}$ be a closed convex cone. Consider a sequence of sets $E_i$ of finite perimeter in $\intt(K)$ such that $v_i=\vol{E_i}\to v$. Then
\begin{equation}
\label{eq:mincone}
\liminf_{i\to\infty}\pp_K(E_i)\ge I_K(v).
\end{equation}
If equality holds, then there is a family of vectors $x_i$ such that $x_i+K\subset K$, and $x_i+E_i$ converges to  a geodesic ball centered at $0$ of volume $v$.
\end{lemma}

\begin{proof}
We assume $K=\rr^k\times\tilde{K}$, where $k\in\nn\cup\{0\}$ and $\tilde{K}$ is  a closed convex cone which contains no lines so that $0$ is an apex of $\tilde{K}$. Inequality \eqref{eq:mincone} follows from $\pp_K(E_i)\ge I_K(v_i)$ and the continuity of $I_K$. Let $B(w)$ be the geodesic ball in $K$ centered at $0$ of volume $w>0$. If equality holds in \eqref{eq:mincone} then
\begin{equation*}
\mu(E_i)=\bigg(\frac{\pp_K(E_i)}{I_K(v_i)}-1\bigg)\to 0.
\end{equation*}
Define $s_i$ by the equality $\vol{B(v_i)}=\vol{s_iB(v)}$. Obviously $s_i\to 1$. By Theorem 1.2 in \cite{FI} there is a sequence of points $x_i\in \rr^k\times \{0\}$ such that
\begin{equation*}
\bigg(\frac{\vol{E_i\,\triangle\, (s_iB(v) +x_i)}}{\vol{E_i}}\bigg)\le C(n,B(v))\,\bigg( \sqrt{\mu(E_i)}+\frac{1}{i} \bigg).
\end{equation*}
Since $\mu(E_i)\to 0$, and $\vol{E_i}\to v>0$, taking limsup we get $\vol{E_i\,\triangle\, (s_iB(v) +x_i)}\to 0$ and so $\vol{(E_i-x_i)\,\triangle\, B(v)}\to 0$, which proves the result. 
\end{proof}

\begin{theorem}
\label{thm:optinsmalvol}
Let $C\subset\rr^{n+1}$ be a convex body. Then
\begin{equation}
\label{eq:limvto0}
\lim_{v\to 0}\frac{I_C(v)}{I_{C_{\min}}(v)}=1.
\end{equation}
Moreover, a rescaling of a sequence of isoperimetric regions of volumes approaching  $0$ has a convergent subsequence in Hausdorff distance to a geodesic ball centered at some vertex in a tangent cone with the smallest solid angle. The same convergence result holds for their free boundaries.
\end{theorem}

\begin{proof}
To prove \eqref{eq:limvto0} we first observe that the invariance of the tangent cone by dilations implies that \eqref{eq:ICleICmin} is valid for every $\la C$ with $\la>0$, i.~e., $I_{\la C}\le I_{C_{\min}}$. So we get
\begin{equation}
\label{eq:optinsmalvol1}
\limsup_{i\to\infty}I_{\la_i C}(v)\le I_{C_{\min}}(v),
\end{equation}
for any sequence $\{\la_i\}_{i\in\nn}$ of positive numbers such that $\la_i\to\infty$ and any $v>0$.

Consider now a sequence $\{E_i\}_{i\in\nn}\subset C$ of isoperimetric regions of volumes $v_i\to 0$ and $p_i\in E_i\cap\ptl C$. Translating the convex set and passing to a subsequence we may assume that $p_i\to 0\in\ptl C$. Let $\la_i=v_i^{-1/(n+1)}$. Then $\la_i\to \infty$ and  $\la_i E_i$ are isoperimetric regions in $\la_i C$ of volume $1$. By Theorem~\ref{thm:connectedness}, the sets $\la_i E_i$ are connected. We claim that
\begin{equation}
\label{eq:claim diam<00}
\sup_{i\in\nn}\diam(\la_iE_i)<\infty.
\end{equation}
If claim holds, since $p_i\to 0$, there is a sequence $\tau_i\to 0$ such that $E_i\subset C\cap\clb(0,\tau_i)$. Let $q\in \intt(C\cap\clb(0,1))$, and consider a solid cone $K_q$ with vertex $q$ such that $0\in \intt(K_q)$ and $K_q\cap C_0\cap \ptl B(0,1)=\emptyset$. Let $s>0$ so that $\clb(0,s)\subset K_q$. Taking $r_i=s^{-1}\tau_i$, $i\in\nn$, we have
\[
r_i^{-1}E_i\subset C\cap\clb(0,r_i^{-1}\tau_i)=C\cap\clb(0,s) \subset K_q.
\]
As the sequence $r_i^{-1}C\cap\clb(0,1)$ converges in Hausdorff distance to $C_0\cap\clb(0,1)$ we construct, using Theorem~\ref{thm:lipschitz}, a family of bilipschitz maps $h_i:r_i^{-1}C\cap \clb(0,1)\to C_0\cap\clb(0,1)$ using the ball $B_q$. So $h_i$ is the identity in $B_q$ and it is extended linearly along the segments leaving from $q$. By construction, the maps $h_i$ have the additional property
\begin{equation}
\label{eq:percone}
P_{C_0}(h_i(r_i^{-1}E_i))=P_{C_0\cap\clb(0,1)}(h_i(r_i^{-1}E_i)).
\end{equation}
So the sequence of bilipschitz maps $g_i: \la_i C\cap \clb(0,\la_i r_i)\to C_0\cap\clb(0,\la_ir_i)$, obtained as in Remark~\ref{rem:lipschitz} with the property $\Lip(h_i)=\Lip(g_i)$ and $\Lip(h_i)=\Lip(g_i^{-1})$ satisfies
\[
P_{C_0}(g_i(\la_iE_i))=P_{C_0\cap\clb(0,\la_ir_i)}(g_i(\la_iE_i)).
\]
This property and Lemma~\ref{lem:bilip} imply
\begin{equation}
\label{eq:mincone1}
\begin{split}
\lim_{i\to\infty}\vol{g_i(\la_iE_i)}&=\lim_{i\to\infty}\vol{\la_iE_i},
\\
\lim_{i\to\infty}\pp_{C_o}(g_i(\la_iE_i))&=\lim_{i\to\infty}\pp_{\la_iC}(\la_iE_i).
\end{split}
\end{equation}
From these equalities, the continuity of $I_{C_0}$, and the fact that $\la_i E_i\subset \la_i C$ are isoperimetric regions of volume $1$, we get
\[
I_{C_0}(1)\le\liminf_{i\to\infty}I_{\la_i C}(1).
\]
combining this with \eqref{eq:optinsmalvol1} and the minimal property of $C_{\min}$ we deduce
\begin{equation*}
\limsup_{i\to\infty}I_{\la_i C}(1)\le I_{C_{\min}}(1)\le I_{C_0}(1)\le\liminf_{i\to\infty}I_{\la_i C}(1).
\end{equation*}
Thus
\begin{equation}
\label{eq:optinsmalvol3}
I_{C_0}(1)=I_{C_{\min}}(1)=\lim_{i\to\infty}I_{\la_i C}(1).
\end{equation}
By \eqref{eq:isopsolang}, we deduce that $C_0$ has minimum solid angle. Finally, from \eqref{eq:optinsmalvol3}, \eqref{lem:link I laC I C}, and the fact that $\la C_{0}=C_0$ we deduce
\[
1=\lim_{i\to\infty}\frac{I_{\la_i C}(1)}{I_{C_{0}}(1)}=
\lim_{i\to\infty}\frac{{\la}_i^nI_C(1/{\la}_i^{n+1})}{{\la}_i^n I_{C_{0}}(1/{\la}_i^{n+1})}=\lim_{i\to\infty}\frac{I_C(v_i)}{I_{C_{0}}(v_i)}.
\]

So it remains to prove \eqref{eq:claim diam<00} to conclude the proof. For this it is enough to prove 
\begin{equation}
\label{eq:uniformmonotonicity}
\pp_{\la_i C}(F_i,B_{\la_i C}(x,r))\ge Mr^n,
\end{equation}
for any $0<r\le 1$, $x\in C$, and any isoperimetric region $F_i\subset\la_i C$ of volume $1$. The constant $M>0$ is independent of $i$. 

To prove \eqref{eq:uniformmonotonicity}, observe first that the constant $M$ in the relative isoperimetric inequality \eqref{eq:isnqgdbl1} is invariant by dilations and, if the factor of dilation is chosen larger than $1$ then the estimate $r\le r_0$ is uniform. The same argument can be applied to a lower Ahlfors constant $\ell_1$. The constant $\ell_2=\omega_{n+1}=\vol{\clb(0,1)}$ is universal and does not depend on the convex set.

Now we modify the proof of Theorem~\ref{thm:leon rigot lem 42} to show that there exists some $\eps>0$, independent of $i$, so that if $h(\la_i E_i,\la_i C,x,r)\le \eps$ then $h(\la_i E_i,\la_i C,x,r/2)=0$, for $0<r\le 1$.

First we treat the case
\[
h(F_i,\la_i C,x,R)=\frac{\vol{F_i\cap B_{\la_i C}(x,R)}}{\vol{B_{\la_i C}(x,R)}}.
\]
By Theorem~\ref{thm:leon rigot lem 42}, since $I_{C}(1)\le I_{\la_i C}(1)$ for all $i\in\nn$, it is enough to take
\begin{equation*}
0<\eps\le\min\bigg\{\frac{1}{\ell_2}, c_2, \frac{I_{C}(1)^{n+1}}{\ell_28^{n+1}}\bigg\}.
\end{equation*}
Now when
\[
h(F_i,\la_i C,x,R)=\frac{\vol{B_{\la_i C}(x,R)\setminus F_i}}{\vol{B_{\la C}(x,R)}},
\]
we proceed as in the proof of Case~1 of Lemma 4.2 in \cite{le-ri}. For $\la_i$ large enough we have $1+\ell_2=\vol{\la_i E_i}+\ell_2<\vol{\la _i C}/2$. As $I_{\la_i C}$ is increasing in the interval $(0,\vol{{\la_i} C}/2]$ the proof of Case~1 in Lemma~4.2 of \cite{le-ri} provides an $\eps>0$ independent of $i$. 

As in Remark~\ref{rem:uniformeps} we conclude the existence of $M>0$ independent of $i$ so that \eqref{eq:uniformmonotonicity} holds.

Now, if $\diam(\la_iE_i)$ is not uniformly bounded, \eqref{eq:uniformmonotonicity} implies that $\pp_{\la_i C}(\la_i E_i)$ is unbounded. But this contradicts the fact that $\pp_{\la_i C}(\la_i E_i)=I_{\la_i C}(1)\le I_{C_{\min}}(1)$ for all $i$.

Finally we prove that $\la_iE_i$ converges to $E$ in Hausdorff distance, where $E\subset C_0$ is a  geodesic ball of volume 1 centered at $0$. By \eqref{eq:mincone1}, $\{g_i(\la_iE_i)\}_{i\in\nn}$ is a minimizing sequence in $C_0$ of volume $1$. By Lemma~\ref{lem:mincone}, translating the whole sequence $\{g_i(\la_iE_i)\}_{i\in\nn}$ if necessary we may assume it is uniformly bounded and so a subsequence of $g_i(\la_iE_i)\to E$ in $L^1(C_0)$. Theorem~\ref{thm:EitoE Haus} implies the Hausdorff convergence of the isoperimetric regions. Theorem~\ref{thm:ptl haus} implies the convergence of the free boundaries.
\end{proof}

From Theorem~\ref{thm:optinsmalvol}  we easily get
\begin{corollary}
\label{cor:CminIC}
Let $C, K\subset \rr^{n+1}$ be convex bodies, with $I_{C_{\min}}>I_{K_{\min}}$. Then for small volumes we have $I_C>I_K$.
\end{corollary} 

For polytopes we are able to show which are the isoperimetric regions for small volumes. The same result holds for any convex set so that there is $r>0$ such that, at every point $p\in\ptl C$ with tangent cone of minimum solid angle we have $B(p,r)\cap C_p=B(p,r)\cap C$.

\begin{theorem}
\label{thm:polytops}
Let $P \subset \rr^{n+1}$ be a convex polytope.  For small volumes the isoperimetric regions in $P$ are geodesic balls centered at vertices with the smallest solid angle.
\end{theorem}

\begin{proof}
Let $\{E_i\}_{i\in\nn}$ be a sequence of isoperimetric regions in $P$ with $\vol{E_i}\to 0$. By Theorem~\ref{thm:optinsmalvol}, a subsequence of $E_i$ is close to some vertex $x$ in $P$. Since $\diam(E_i)\to 0$ we can suppose that, for small enough volumes, the sets $E_i$ are also subsets of the tangent cone $P_x$ and they are isoperimetric regions in $P_x$. By \cite{FI} the only isoperimetric regions in this cone are the geodesic balls centered at $x$. These geodesic balls are also subsets of $P$.\end{proof} 

\begin{remark}
\label{rem:Malchiodi's student}
In \cite{fall} Fall considered the partitioning problem of a domain with smooth boundary in a smooth Riemannian manifold. He showed that, for small enough volume, the isoperimetric regions are concentrated near the maxima of the mean curvature function and that they are asymptotic to half-spheres. The techniques used in this paper are similar to the ones used by Nardulli \cite{nardulli} in his study of isoperimetric regions of small volume in compact Riemannian manifolds. See also \cite[Thm.~2.2]{MR1803220}.
\end{remark}

\begin{proposition}
Let $C\subset\rr^{n+1}$ be a convex body and $\{E_i\}_{i\in\nn}$ a sequence of isoperimetric regions with $\vol{E_i}\to 0$. Assume that $0\in\ptl C$ and that $C_0$ is a tangent cone with the smallest solid angle.  Let $\la_i>0$ be so that $\vol{\la_i E_i} =1$, and let $E\subset C_0$ be the geodesic ball in $C_0$ centered at $0$ of volume $1$. Then, for every $x\in\ptl E\cap\intt(C_0)$ so that $B(x,r)\subset\intt(C_0)$, the boundary $\ptl \la_iE_i\cap B(x,r)$ is a smooth graph with constant mean curvature for $i$ large enough.
\end{proposition}

\begin{proof}
We use Allard's Regularity Theorem for rectifiable varifolds, see \cite{MR0307015}, \cite{simon}.

Assume $\{E_i\}_{i\in\nn}$ is a sequence of isoperimetric regions of volumes $v_i\to 0$, and that $0\in\ptl C$ is an accumulation point of points in $E_i$. We rescale so that $\vol{\la_iE_i}=1$, project to $C_0$ (by means of the mapping $g_i$), and rescale again to get a minimizing sequence $F_i$ in $C_0$ of volume $1$. The sequence $\{F_i\}_{i\in\nn}$ converges in $L^1(C_0)$ by Lemma~\ref{lem:mincone}.

If $v_i=\vol{E_i}\to 0$ then $\la_i=v_i^{-1/(n+1)}$. Let $H_i$ be the constant mean curvature of the reduced boundary of $E_i$. Then the mean curvature of the reduced boundary of $\la_iE_i$ is $\frac{1}{\la_i}H_i=v_i^{1/(n+1)}H_i$. Let us check that these values are uniformly bounded.

From \eqref{eq:I_C(v)> cv.} we get
\begin{equation}
\label{eq:Hiarebounded1}
I_C(v)\ge m v^{n/(n+1)},
\end{equation}
for all $0<v<\frac{\vol{C}}{2}$ with $m=I_C(\vol{C}/2)/(\vol{C}/2)^{n/(n+1)}$. We also have
\begin{equation}
\label{eq:Hiarebounded2}
I_C^{(n+1)/n}(v)\le M v
\end{equation}
for all $0<v<\vol{C}$. Here $M$ can be chosen as a power of the isoperimetric constant of $C_{\min}$ or $\mathbb{H}^{n+1}$ since $I_C\le I_{C_{\min}}\le I_H$ by Proposition~\ref{prp:ICleICmin} and Remark~\ref{rem:half-plane}.
Since $Y_C=I_C^{(n+1)/n}$ is concave, given $h>0$ small enough, using \eqref{eq:Hiarebounded2} we have
\begin{equation*}
\frac{Y_C(v)-Y_C(v-h)}{h}\le \frac{Y_C(v)}{v}\le M.
\end{equation*}
Taking limits when $h\to 0$ we get
\begin{equation*}
(Y_C)_{-}'(v)\le M,
\end{equation*}
for all $0<v<\vol{C}$. By the chain rule
\begin{equation*}
\bigg(\frac{n+1}{n}\bigg)\,I_C^{1/n}(v)\,(I_C)_{-}'(v)=(Y_C)_{-}'(v)\le M.
\end{equation*}
Since the mean curvature $H$ of any isoperimetric region of volume $v$ satisfies $H\le (I_C)_{-}'(v)$, using \eqref{eq:Hiarebounded1} we have
\begin{equation*}
\bigg(\frac{n+1}{n}\bigg)m^{1/n}v^{1/(n+1)}H\le \bigg(\frac{n+1}{n}\bigg)I_C^{1/n}(v)(I_C)_{-}'(v)=(Y_C)_{-}'(v)\le M
\end{equation*}
So the quantity $v^{1/(n+1)}H$  is uniformly bounded for any $0<v<\vol{C}$. This implies that the constant mean curvature of the reduced boundary of the regions $\la_iE_i$ is uniformly bounded.
\end{proof}

\bibliography{convex}

\end{document}